\documentclass[12pt]{amsart}
\usepackage{amssymb,amsmath,amsthm}
\usepackage{srcltx}
\usepackage{soul}
\usepackage{cancel}
\usepackage{color}

\newtheorem{proposition}{Proposition}[section]
\newtheorem{theorem}[proposition]{Theorem}

\newtheorem{lemma}[proposition]{Lemma}
\theoremstyle{definition}
\newtheorem{definition}[proposition]{Definition}
\newtheorem{remark}[proposition]{Remark}

\numberwithin{equation}{section}

\def\u{{\boldsymbol{u}}}
\def\v{{\boldsymbol{v}}}
\def\n{{\boldsymbol{n}}}
\def\HD{{\boldsymbol{V}}}

\def\e{{\rm e}}
\def\eps{\varepsilon}
\def\d{{\rm d}}

\def\ddt{\frac{\d}{\d t}}
\def\R {\mathbb{R}}
\def\N {\mathbb{N}}

\def\H {{\rm H}}

\def\C {{\mathcal C}}

\def\S {{\mathcal S}}

\def\Q {{\mathcal Q}}

\def \l {\langle}
\def \r {\rangle}

\def \and{\qquad\text{and}\qquad}

\def \x {\boldsymbol{x}}
\def \au {\rm}
\def \ti {\it}
\def \jou {\rm}

\def \no#1#2#3 {{\bf #1} (#3), #2.}
\def \eds#1#2#3 {#1, #2, #3.}
\begin{document}

\title[]
{On the Cahn-Hilliard-Brinkman system}
\author[S. Bosia, M. Conti, M. Grasselli]
{Stefano Bosia, Monica Conti, Maurizio Grasselli}
\address{Politecnico di Milano - Dipartimento di Matematica
\newline\indent
Via E. Bonardi 9,  20133 Milano, Italy}
\email{stefano.bosia@polimi.it}
\email{monica.conti@polimi.it}
\email{maurizio.grasselli@polimi.it}

\begin{abstract}
We consider a diffuse interface model for phase separation of an isothermal incompressible
binary fluid in a Brinkman porous medium. The coupled  system consists of a convective
Cahn-Hilliard equation for the phase field $\phi$, i.e., the difference of the (relative) concentrations of the
two phases, coupled with a modified Darcy equation proposed by H.C. Brinkman in 1947
for the fluid velocity $\u$.
This equation incorporates a diffuse interface surface force
proportional to $\phi\nabla\mu$, where  $\mu$ is the so-called chemical potential.
We  analyze the well-posedness of the resulting Cahn-Hilliard-Brinkman (CHB) system for $(\phi,\u)$. Then we establish the existence of
a global attractor and the convergence of a given (weak) solution to a single equilibrium via {\L}ojasiewicz-Simon inequality.
Furthermore,  we study the behavior of the solutions
as the viscosity goes to zero, that is, when the CHB system approaches the Cahn-Hilliard-Hele-Shaw (CHHS) system.
We first prove the existence of a weak solution to the  CHHS system as limit of CHB solutions. Then, in dimension two, we estimate the difference
of the solutions to CHB and CHHS systems in terms of the viscosity constant appearing in CHB.

\end{abstract}

\maketitle


%%%%%%%%%%%%%%%%%%%%%%%%%%%%%%%%%%%%%%%%%%%%%%%%%%%%%%%%%%%%%%%%%%%

\section{Introduction}

\noindent
The so-called Brinkman equation was proposed by H.C. Brinkman in \cite{Br} as a modified Darcy's law
in order to describe the flow through a porous mass. If we assume that the incompressible
fluid occupies a bounded domain $\Omega\subset\R^d$, $d=2,3$, for any time $t\in (0,T)$, $T>0$,
 the Brinkman equation for the (divergence free)
fluid velocity $\u$ reads
$$
-\nabla \cdot [\nu D(\u)]+\eta \u=-\nabla p,
$$
in $\Omega\times (0,T)$.
Here $2D(\u) =  \nabla \u+(\nabla \u)^{tr}$, $\nu>0$ is the viscosity, $\eta>0$ the fluid permeability
and $p$ is the fluid pressure.

More recently, a diffuse interface variant of Brinkman equation has been proposed to model phase separation of incompressible binary fluids
in a porous medium (see \cite{NYT}). Let us suppose that both the fluids have equal constant density and indicate by
$\phi$ the difference of the fluid (relative) concentrations. Denoting by $\u$  the (averaged) fluid velocity, the resulting model is the following
\begin{align}
\label{B1}
&\partial_t \phi+\nabla\cdot (\phi\u)=\nabla\cdot (M\nabla \mu),\\
\label{Bmu}
&\mu=-\eps \Delta \phi +\frac1\eps f(\phi),\\
\label{B2}
&-\nabla \cdot [\nu D(\u)]+\eta \u=-\nabla p-\gamma \phi \nabla \mu,\\
\label{Bdiv}
&\nabla \cdot \u=0,
\end{align}
in $\Omega\times(0,T)$. Here $M>0$ stands for the mobility, $\eps>0$ is related to the diffuse interface thickness,
$f$ is the derivative of a double well potential describing phase separation, and $\gamma>0$ is a surface tension parameter.

This model consists of a convective Cahn-Hilliard equation \eqref{B1}-\eqref{Bmu} coupled with the Brinkman equation
through the surface tension force $\gamma \phi \nabla \mu$. For this reason \eqref{B1}-\eqref{Bdiv} has been called
Cahn-Hilliard-Brinkman (CHB) system. Such a system belongs to a class of diffuse interface models
which are used to describe the behavior of multi-phase fluids. We recall, in particular, the Cahn-Hilliard-Navier-Stokes
system which has been investigated in several papers
(see, e.g., \cite{Abels2009, Abels2-2009, Boyer1999,CaoGal2012,GalGra2010,GalGra2-2010,GalGra2011,LiuShen2003,Starovoitov1997,
ZhaoWuHuang2009, ZhouFan2013},
cf. also \cite{Kim2012} for a recent review on modeling and numerics).

CHB system has recently been analyzed from the numerical viewpoint in \cite{CSW} (see also \cite{DFW}). More precisely,
the authors have considered system \eqref{B1}-\eqref{Bdiv} with $M$, $\nu$ and $\eta$ possibly depending on $\phi$ and
endowed with the boundary and initial conditions
\begin{align}
\label{BCu}
&\u|_{\partial\Omega}=\mathbf{0}, \quad\text{ on }\partial\Omega\times(0,T),
\\\label{BCfi}
&\partial_\n \phi=\partial_\n\mu=0, \quad\text{ on }\partial\Omega\times(0,T),
\\\label{IC}
&\phi(0)=\phi_0,
\end{align}
where $\phi_0:\Omega\to\R$ is a given function. Here $\n$ stands for the outward normal vector to $\partial\Omega$
which is supposed to be smooth enough.

The main goal of this contribution is to establish some theoretical results
on \eqref{B1}-\eqref{IC}, in the case when $M$, $\nu$ and $\eta$ are constant.
First of all we analyze the well-posedness of the problem, proving  the global existence and uniqueness
of a weak solution and its continuous dependence on the initial datum.
Secondly, we study  the longterm behavior of the CHB system as a dissipative dynamical system by proving the existence of
a global attractor. Then we investigate the long-time dynamics of any given weak solution by showing
that each trajectory  does converge to a unique stationary state, with an explicit convergence rate.
Our results includes the case $\eta=0$ (see \cite{Schmuck2013} and references therein).

In the second part of the paper we analyze the behavior of solutions when $\nu$ goes to zero. Observe that when $\nu=0$ system \eqref{B1}-\eqref{Bdiv}
becomes the so-called Cahn-Hilliard-Hele-Shaw (CHHS) model. This is a particularly challenging problem which
finds applications in tumor growth dynamics (see, e.g., \cite{LTZ} and its references) and has been recently studied from the theoretical viewpoint in \cite{LTZ,WW,WZ} (see also \cite{FW,LLG1,LLG2} and references therein).

We are able to prove that there is a global weak solution to CHHS system which is the limit of
solutions to CHB system with \eqref{BCu}-\eqref{IC} (compare with \cite[Thm.2.4]{FW}).
Notice that uniqueness of weak solutions is still an open problem. On the contrary, a \emph{strong} solution is unique,
but, if $d=3$, only local existence is known so far unless the initial datum is a small perturbation of
a suitable constant state (see \cite{LTZ}).

In dimension two, we also provide an estimate of the difference of (strong) solutions to CHB and CHHS systems with respect to $\nu$.

The plan of this paper goes as follows. In the next section we state the main results along with some notation and basic tools. Section~\ref{S:basic_estimates}
is devoted to prove certain a priori estimates. Then, in Section~\ref{wpCHB}, we establish the well-posedness of problem \eqref{B1}-\eqref{IC}
and a global dissipative estimate. In Section~\ref{S:Higher_order} we obtain some higher-order estimates  which are helpful to prove the existence of
the global attractor as well as to show, in Section~\ref{convequil}, the convergence to the equilibrium
of a given weak solution. Finally, in Section~\ref{S:nu_to_0}, we analyze what happens when $\nu$ goes to zero,
 while in Section~\ref{CHB2D} we estimate the difference of (strong) solutions to CHB and CHHS systems.

\section{Preliminaries and main results}
\label{Sec2}

Here we list our assumptions on $f$ and the potential $F(s):=\int_0^s f(y)\,\d y$ and we introduce some notation.
Then we state our main results. This requires to formulate our problems rigorously.
We also recall a pair of Gronwall-type lemmas.

\subsection*{Assumptions on $F$ and $f$}
We assume that $f\in \C^1(\R)$, with $f(0)=0$, is such that
\begin{equation}
\label{GROW}
|f(s)|\leq c(1+|s|^3),
\end{equation}
and
\begin{align}
\label{DISS1}
F(s)\geq -c,
\end{align}
for all $s\in\mathbb{R}$ and some $c>0$.
In the course of the investigation we
shall need further assumptions such as
\begin{equation}
\label{GROW2}
|f'(s)-f'(t)|\leq c|s-t|(1+|s|+|t|),
\end{equation}
or the stronger condition $f\in \C^2(\R)$ such that
\begin{equation}
\label{GROW3}
|f''(s)|\leq c(1+|s|).
\end{equation}
We shall also make use of the following
dissipation condition
\begin{equation}
\label{DISS0}
\inf_{s\in\R} f'(s)>-\infty.
\end{equation}

\medskip
\noindent
A typical example of (regular) double well potential is
\begin{equation}
\label{doublewell}
F(s)=(s^2-1)^2,
\end{equation}
which complies with \eqref{GROW}-\eqref{DISS0}. More generally, one can take
a fourth degree polynomial with positive leading coefficient.

\subsection*{Functional spaces}
Let $\Omega\subset\R^d$, $d=2,3$, be either a smooth bounded connected domain or a convex polygonal or polyhedral domain.
For any positive integer $r$, let
$H^r(\Omega)=W^{r,2}(\Omega)$,
 the usual Sobolev space, and  denote the norm $\|\cdot\|_{W^{r,2}(\Omega)}$ by $\|\cdot\|_r$.
Throughout the paper, we set $\H=L^2(\Omega)$,
$$ V=\overline{\{\phi\in \v \in \C^{\infty}(\overline{\Omega})\,:\,\partial_\n \phi=0 \text{ on }\partial\Omega\}}^{H^1(\Omega)}\quad\text{and}\quad    \H^r=H^r(\Omega)\cap V, $$
endowed with the norm $\|\cdot\|_r$.
Similarly, we denote the norm $\|\cdot\|_{L^2}$ by $\|\cdot\|$.
The shorthand $\l\cdot,\cdot\r$ will  stand both for the scalar product in $\H$ and for the duality product between
$\H^r$ and its dual space $\H^{-r}$.
 The same symbols will also be used for the scalar product and norm in spaces of vector-valued elements.

\noindent Besides, let $\mathcal{V}$ be the space of divergence-free test functions defined by
\begin{equation*}
    \mathcal{V} = \{ \v \in \C^{\infty}_{0}(\Omega, \mathbb{R}^{3}) \, : \, \nabla \cdot \v = 0 \}.
\end{equation*}
We shall use the following spaces
$$
    {\boldsymbol H}    = \overline{\mathcal{V}}^{(\H)^{3}} \quad\text{ and }\quad
    \HD                = \overline{\mathcal{V}}^{(\H^1)^{3}}.
$$
In particular we recall that if $\v \in \HD$ then $\v |_{\partial \Omega} = \mathbf{0}$ and if $\v \in {\boldsymbol H}$ then $\v \cdot \n = \mathbf{0}$ on $\partial \Omega$ (see, e.g., \cite[Chapter~I]{Temam2001}).

\subsection*{Notation}
Without loss of generality we will set $M=\eps=\gamma = 1$.
Throughout the paper, $c\geq 0$ will stand for a generic constant and
$\Q(\cdot)$ for a generic positive increasing function.
%%%%%%%%%%%%%%%%%%%%%%%%%%%%%%

%%%%%%%%%%%%%%%%%%%%%%%%%%%%%%
\subsection{Statement of the main results}
Let us introduce the definition of weak solution to the CHB system with boundary and initial conditions \eqref{BCu}-\eqref{IC}.
\begin{definition}
\label{d-solution}
Let $\nu>0$, $\phi_0\in \H^1$ and $T>0$ be given.
A pair $(\phi,\u)$ is a (weak) solution to system \eqref{B1}-\eqref{Bdiv} endowed with \eqref{BCu}-\eqref{IC} if
$$\phi\in \C([0,T], \H^1)\cap L^2(0,T;\H^3) $$
satisfies
\begin{align}
\label{B1-w}
&\l \partial_t\phi(t), w\r+\l \nabla\cdot(\phi(t)\u(t)),w\r + \l \nabla\mu(t),\nabla w\r =0,\quad \forall w\in \H^1, \quad \text{a.e.\ $t \in [0,T]$},\\\nonumber
&\partial_\n \phi=0, \quad\text{ a.e. on }\partial\Omega\times(0,T),\\\nonumber
&\phi|_{t=0}=\phi_0, \quad\text{ a.e. in }\Omega,
\end{align}
with $\mu\in L^2(0,T;\H^1)$ given by \eqref{Bmu}
and
$$\u\in L^2(0,T;\HD)$$
fulfills
\begin{equation}
\label{B2-w}
\nu\l \nabla\u(t),\nabla \v\r+\eta \l \u(t),\v\r =-\l \phi(t)\nabla \mu(t), \v\r,\quad \forall\v\in \HD, \quad \text{a.e.\ $t \in [0,T]$}.
\end{equation}
\end{definition}

\smallskip

\begin{remark} It is straightforward to observe that any weak solution satisfies mass conservation, namely,
\begin{equation}
\label{masscons}
\l\phi(t)\r=\l\phi_0\r,\quad \forall t\geq 0,
\end{equation}
where
$$\l\phi(t)\r:=\frac{1}{|\Omega|}\int_{\Omega} \phi(\x,t)\,\d \x.$$
\end{remark}

\begin{remark}
    As we shall see in Section~\ref{S:basic_estimates}, the regularity assumed in Definition~\ref{d-solution} yields
            $$\nabla\cdot (\phi\u)\in L^{2}(0,T;\H^{-1}),$$
            so that $\partial_{t} \phi \in L^{2}(0,T;\H^{-1})$ by comparison.
            Besides, we have
            $$\phi\nabla\mu\in L^{8/5}(0,T;{\boldsymbol H}) \cap L^{2}(0,T;\HD^{\ast}).$$
\end{remark}

\begin{remark}
\label{pressure}
    As usual the pressure term is dropped in the weak formulation of the Stokes problem.
    Indeed, the pressure can be recovered (up to a constant) thanks to a classical result (see, for instance, \cite[Theorem~I.1.4]{Temam2001}).
In particular,  since
$$ \S= \nu \Delta\u-\eta\u+\phi\nabla \mu\in L^{2}(0,T;\HD^{\ast}),$$
we know that there exists a (unique up to an additive function of $t$ only) function
$p \in L^{2}(0,T;\H)$ satisfying $\nabla p=\S$.
\end{remark}

Global existence and uniqueness of a weak solution is given by

\begin{theorem}
\label{t:wellposed}
Let $\nu>0$, $\eta\geq 0$ and $f$ satisfy \eqref{GROW}-\eqref{DISS1}.
Let $\phi_0\in \H^1$ be given. Then, for every $T>0$, there exists a pair
$(\phi,\u)$ which is a solution to the CHB system according to Definition \ref{d-solution}.
If \eqref{GROW2} holds, then the weak solution is unique.
\end{theorem}

We also have  continuous dependence estimates.

\begin{theorem}
\label{t:dipco}
Let  $\nu>0$, $\eta>0$. Under the same assumptions of Theorem \ref{t:wellposed},
if $(\phi_1,\u_1)$ and $(\phi_2,\u_2)$ are two weak solutions to the CHB system such that
$\l\phi_1(0)\r=\l\phi_2(0)\r$,
then, for every $T>0$, there exists $C_T>0$ depending on $R=\max\{\|\phi_1(0)\|_1,\|\phi_2(0)\|_1\}$
such that the following continuous dependence estimates hold
\begin{equation}
\label{dipcont}
\|\phi_1(t)-\phi_2(t)\|_1^2\leq \|\phi_1(0)-\phi_2(0)\|_1^2 \e^{C_T/\sqrt{\nu}},
\end{equation}
and
\begin{equation}
\label{dipcont2}
\int_{0}^{t} \| \u_1(t)-\u_2(t) \|_{1}^{2} \leq \|\phi_1(0)-\phi_2(0)\|_1^2 \left( 1 + C_T\e^{C_T/\sqrt{\nu}}\right),
\end{equation}
for every $t\in [0,T]$.
\end{theorem}

\begin{remark}
\label{r:dc0}
    In the case $\eta=0$ (cf. \cite{Schmuck2013}) the same continuous dependence estimates hold by replacing $\sqrt{\nu}$ with $\nu$ in \eqref{dipcont} and \eqref{dipcont2}.
\end{remark}

The next result shows that any weak solution converges to a single stationary state as time goes to infinity.

\begin{theorem}
\label{LS}
Let $\nu>0$, $\eta\geq 0$ and
let $f$ be real analytic satisfying \eqref{DISS1}-\eqref{DISS0}.
For every fixed $\phi_0\in \H^1$,
the global solution $\phi$ originating from $\phi_0$ converges to an equilibrium $\phi^\star$ as $t\to\infty$, with the following convergence rate
\begin{equation}
\label{convrate}
\|\phi(t)-\phi^\star\|_1 \leq \frac{c_\nu}{(1+t)^{\theta/(1-2\theta)}},\quad\forall t\geq t^\ast,
\end{equation}
for some $\theta=\theta(\phi^\star)\in (0, \tfrac{1}{2})$,
$c_\nu=c_\nu(\|\phi_0\|_1)\geq 0$ and $t^\ast>0$.
Here $\phi^\star\in \H^2$ is a
solution to the stationary system
$$
-\Delta z+ f(z)=const\, \text{ in }\Omega,\qquad\partial_\n z=0 \text{ on }\partial\Omega,\qquad\l z\r=\l\phi_0\r.
$$
Furthermore, the velocity field $\u$ vanishes and satisfies
\begin{equation}
\label{convrate-u}
\|\u(t)\|_1 \leq \frac{c_\nu}{(1+t)^{\theta/4(1-2\theta)}},\quad\forall t\geq t^\ast.
\end{equation}
Here $c_\nu\to\infty$ as $\nu\to 0$.
\end{theorem}

Let us now introduce the definition of weak solution to the CHHS system endowed with \eqref{BCfi}-\eqref{IC} and
\begin{equation}
\label{BCuHS}
\u\cdot\n =\mathbf{0}, \quad\text{ on }\partial\Omega\times(0,T).
\end{equation}

\begin{definition}
\label{d-solutionHS}
Let $\phi_0\in \H^{1}$ and $T>0$ be given.
A pair $(\phi,\u)$ is a (weak) solution to the CHHS system endowed with \eqref{BCfi}-\eqref{IC} and \eqref{BCuHS} if
$$\phi\in \C_{\mathbf{w}}([0,T], \H^{1})\cap L^2(0,T;\H^3)$$
satisfies
\begin{align*}
&\l \partial_t\phi(t), w\r+\l \nabla\cdot(\phi(t)\u(t)),w\r+\l \nabla\mu(t),\nabla w\r=0,\quad \forall w\in \H^{1}, \quad \text{a.e.\ $t \in [0,T]$},\\
&\partial_\n \phi=0, \quad\text{ a.e. on }\partial\Omega\times(0,T),\\
&\phi|_{t=0}=\phi_0, \quad\text{ a.e. in }\Omega,
\end{align*}
with $\mu\in L^2(0,T;\H^{1})$ given by \eqref{Bmu}
and
$$\u\in L^2(0,T;{\boldsymbol H})$$
fulfills
$$\eta \l \u(t),\v\r =-\l \phi(t)\nabla \mu(t), \v\r,\quad \forall\v\in \HD, \quad \text{a.e.\ $t \in [0,T]$}.$$
\end{definition}

\begin{remark}
    It is worth noting that the regularity assumed in Definition~\ref{d-solutionHS} yields
        $$\nabla\cdot (\phi\u)\in L^{8/5}(0,T;\H^{-1})\qquad\text{whence}\qquad \partial_{t} \phi \in L^{8/5}(0,T;\H^{-1}).$$
\end{remark}

The following theorem says that a weak solution to the CHHS system can be found as a limit of solutions to CHB system
as viscosity vanishes.
\begin{theorem}
\label{t:nuto0}
Let $\eta>0$ and let $f$ satisfy \eqref{GROW}-\eqref{DISS1}.
 For $\phi_0\in \H^1$ let $\{\nu_n\}_{n\in\N}$ be a sequence of positive numbers such that $\nu_n\to 0$ as $n\to\infty$.
 Let $(\phi_n,\u_n)$ be the sequence of weak solutions corresponding to the CHB system with $\nu=\nu_n$ originating from $\phi_0$.
Then, up to a subsequence, $(\phi_n,\u_n)$ converges to a weak solution $(\phi,\u)$ to the CHHS system according to
Definition~\ref{d-solutionHS}
in the following sense:
\begin{align*}
&\phi_n \to \phi \quad \text{ weakly in } L^2(0,T;\H^3)  \text{ and strongly in } L^2(0,T;\H^2),\\
&\u_n\to \u \quad \text{ weakly in } L^2(0,T;{\boldsymbol H}).
\end{align*}
\end{theorem}

Finally, in dimension two, we state a result about the estimate of the difference between a solution to the BCH system and
a solution to the CHHS system.
Indeed, it is known from \cite{LTZ} that the CHHS system endowed with \eqref{BCfi}-\eqref{IC} and \eqref{BCuHS}
admits a unique \emph{strong} solution provided that $\phi_0\in\H^2$, which is also \emph{global}  when $d=2$.
In this case, we have the following result
\begin{theorem}
\label{t:close}
Let $d=2$ and $\eta>0$. Let $f$ satisfy \eqref{DISS1}-\eqref{GROW2}.
Take $\phi_0^\nu,\phi_0\in\H^2$ such that $\l\phi_0^\nu\r=\l\phi_0\r$ and set
$$R:=\sup_{\nu> 0}\{\|\phi_0^\nu\|_2,\|\phi_0\|_2\}<\infty.$$
Let $(\phi_\nu,\u_{\nu})$ be the unique weak solution to the CHB system with $\nu>0$,
originating from $\phi_0^\nu$, and $(\phi, \u)$ the solution to the CHHS system with initial datum $\phi_0$.
Then, for every $T>0$, there exists $C_T>0$ (depending only on $R$) such that
$$\|\phi_\nu(t)-\phi(t)\|_1^2+\int_0^t\|\u_\nu(y)-\u(y)\|^2\,\d y\leq
\|\phi_0^\nu-\phi_0\|_1^2\e^{C_T}+C_T\nu^{1/2},\quad \forall t\in [0,T].$$
In particular, if $\phi_0^\nu=\phi_0$, then
$$\phi_\nu\to \phi \quad\text{in }L^\infty(0,T;\H^1)\;\text{ as }\nu\to 0,$$
for all $T>0$.
\end{theorem}
%%%%%%%%%%%%%%%%%%%%%%%%%%%%%%%%%%%%%%%%%

%%%%%%%%%%%%%%%%%%%%%%%%%%%%%%%%%%%%%%%%%
\subsection{Basic inequalities}
We will exploit the classical inequalities due to Sobolev, Gagliardo and Nirenberg, Agmon and Poincar\'{e}, respectively, which are standard (see, e.g., \cite{Tartar2007,Temam2001}).

We also need a pair of Gronwall-type inequalities.
The uniform Gronwall lemma (\cite[Section~1.1.3]{Temam1997}), namely,

\begin{lemma}
\label{gen-gronw}
Let $\psi_0$ be an absolutely continuous nonnegative function and $\psi_1, \psi_2$ be
two nonnegative functions satisfying, almost everywhere in $\R^+$,
the differential inequality
$$
\ddt \psi_0\leq \psi_0\psi_1+\psi_2.
$$
Assume also that
$$
\sup_{t\geq0}\int_t^{t+r}\psi_{\imath}(\tau) d\tau\leq m_{i}, \quad i= 0,1,2,
$$
for some positive constants $m_{\imath}$ and $r>0$.
Then,
$$\psi_0(t+r) \leq \left( \frac{m_0}{r} + m_2 \right) e^{m_1},\quad \forall t\geq0.$$
\end{lemma}

The following differential Gronwall lemma whose proof is elementar.
\begin{lemma}
\label{gr}
Let $\psi:[t^\star,\infty)\to\R$ be an absolutely continuous
function, which fulfills for almost every $t\geq t^\star$
the differential inequality
    \begin{equation*}
        \ddt \psi(t) + \alpha \psi(t) \leq (1+t)^{-\beta},
    \end{equation*}
for some $\alpha>0$ and $\beta>0$.
    Then, there exists $c>0$ such that, for every sufficiently large time $t$
        \begin{equation*}
        \psi(t) \leq c(1+\psi(t^\star)) (1+t)^{-\beta}.
    \end{equation*}
\end{lemma}
%%%%%%%%%%%%%%%%%%%%%%

%%%%%%%%%%%%%%%%%%%%%%
\section{Basic estimates}\label{S:basic_estimates}
In this section we let $\phi_0\in \H^1$ and we denote by
$(\phi,\u)$ a weak solution to the CHB system originating from $\phi_0$.
Our aim is to prove a number of a priori estimates for $(\phi,\u)$.

To this aim, in the following we denote by $\Q(\cdot)$ a generic increasing and positive function
which is \emph{independent of $\nu$}. All the energy estimates are formal but they can be performed rigorously
within a Galerkin approximation scheme (see Section \ref{wpCHB} for references).

\subsection{Energy estimates}$\,$
\begin{lemma}
For any given $R>0$ the following inequality holds
\begin{equation}
\label{e-base}
\|\phi(t)\|_1^2+\int_0^\infty\big(\|\nabla\mu(y)\|^2+\eta\|\u(y)\|^2\big)\d y+\nu\int_0^\infty\|\nabla \u(y)\|^2\d y\leq \Q(R),
\end{equation}
for every initial datum $\phi_0$ with $\|\phi_0\|_1\leq R$.
Besides, for every $T>0$, we have
\begin{equation}
\label{e2-base}
\int_0^T\big(\|\mu(y)\|_1^2+\|\phi(y)\|_3^2\big)\d y\leq \Q_T(R),
\end{equation}
for some increasing positive function $\Q_T$ depending on $T$.
\end{lemma}

\begin{proof}
Taking $w=\mu$ in \eqref{B1-w} and $\v=\u$ in \eqref{B2-w},
and summing up the resulting equalities, we have
\begin{equation}
\label{s-base}
\ddt\Big(\frac12 \|\nabla\phi\|^2+\l F(\phi),1\r\Big)+\|\nabla\mu\|^2+\nu\|\nabla\u\|^2
+\eta\|\u\|^2=0.
\end{equation}
In light of~\eqref{GROW}, this provides
$$\|\phi(t)\|_1\leq \|\phi_0\|_1 + 2\l F(\phi_{0}),1\r\leq \Q(R) ,\quad \forall t\geq 0.$$
A subsequent integration in time of \eqref{s-base} completes the proof of
\eqref{e-base}.\\
Now, multiplying \eqref{Bmu} in $\H$ by the constant function $1$, we get
$$\l \mu,1\r=\l f(\phi),1\r,$$
which, by \eqref{GROW}, gives
$$\l \mu\r\leq c(1+\int_\Omega|\phi|^3)\leq c(1+\|\phi\|_1^3)\leq \Q(R).$$
Thanks to \eqref{e-base} we obtain, for every $T>0,$
$$\int_0^T \|\mu(y)\|_1^2\,\d y\leq \Q_T(R),$$
hence $\mu\in L^2(0,T;\H^1).$
Let us now
multiply \eqref{Bmu} by $-\Delta^2\phi$ in $\H$. This yields
$$\l \nabla \mu,\nabla\Delta\phi\r=-\|\nabla\Delta\phi\|^2+\l f'(\phi)\nabla \phi,\nabla\Delta\phi\r.$$
On the other hand, recalling \eqref{GROW} and~\eqref{e-base}, we have
\begin{align*}
\l f'(\phi)\nabla\phi,\nabla\Delta\phi\r&\leq \|f'(\phi)\|_{L^{3}} \|\nabla \phi\|_{L^6}\|\nabla\Delta\phi\|\\
&\leq \Q(R) \|\nabla \phi\|^{1/2} \|\nabla\Delta\phi\|^{1/2} \|\nabla\Delta\phi\|\\
&\leq \Q(R)+\frac{1}{4} \|\nabla\Delta\phi\|^2,
\end{align*}
which entails
$$\frac12\|\nabla\Delta\phi\|^2\leq \|\nabla \mu\|^2+\Q(R).$$
Owing to \eqref{e-base}, we find
\begin{equation}
\label{L2-H3}
\int_0^T\|\phi(y)\|_3^2\,\d y\leq \Q_T(R),
\end{equation}
so that
$\phi \in L^{2}(0,T; \H^{3})$, completing the proof of \eqref{e2-base}.
\end{proof}

\begin{remark}
Since $\phi\in L^\infty(\R^+;\H^1)\cap L^2(0,T;\H^3)$, we easily get by interpolation
\begin{align*}
\int_0^T \|\phi\|_2^p\,\d t\leq\int_0^T \|\phi\|_1^{p/2}\|\phi\|_3^{p/2}\,\d t\leq c\int_0^T \|\phi\|_3^{p/2}\,\d t < \infty
\quad\mbox{ if }\;p \leq 4.
\end{align*}
Thus
\begin{equation}\label{E:l4h2}
\int_0^T \|\phi(y)\|_2^4\,\d y\leq \Q_T(R),
\end{equation}
that is, $\phi\in L^4(0,T;\H^2)$.
\end{remark}

\subsection{Further Estimates}$\,$

\medskip
\noindent  \textbf{The term $\nabla\cdot (\phi\u)$}. For $w\in \H^1$, using the
Agmon inequality and interpolation, we compute
\begin{align*}
\l \nabla\cdot (\phi\u),w \r&= \l\phi\u,\nabla w \r\leq \|\nabla w\|\|\u\|\|\phi\|_{L^\infty}\\
&\leq \|\nabla w\|\|\u\|\|\phi\|_1^{3/4}\|\phi\|_3^{1/4}\leq \Q(R)\|\nabla w\|\|\u\|\|\phi\|_3^{1/4}.
\end{align*}
This implies
\begin{align*}
\Big|\int_0^T \l \nabla\cdot (\phi\u),w \r\,\d t\Big|
&\leq \Q(R) \int_0^T\|\nabla w\|\|\u\|\|\phi\|_3^{1/4}\,\d t\\
&\leq \Q(R) \Big(\int_0^T\|\nabla w\|^{8/3} \d t\Big)^{3/8} \Big(\int_0^T\|\u\|^{2}\,\d t\Big)^{1/2} \Big(\int_0^T\|\phi\|_3^{2}\,\d t\Big)^{1/8}.
\end{align*}
As a consequence, invoking the fact that $\u\in L^2(0,T;{\boldsymbol H})$ and $\phi\in L^2(0,T;\H^3)$,
we get
$$
\Big|\int_0^T \l \nabla\cdot (\phi\u),w \r\,\d t\Big|\leq \Q_T(R) \Big(\int_0^T\|\nabla w\|^{8/3}\d t\Big)^{3/8},
$$
which gives
$$\nabla\cdot (\phi\u)\in L^{8/5}(0,T; \H^{-1}).$$
We stress that this control is independent of $\nu$.
Exploiting the $\nu$-dependent estimate $\u\in L^2(0,T; \HD)$ we can improve the previous estimate. Indeed, we have
\begin{align*}
\l \nabla\cdot (\phi\u),w \r= \l\phi\u,\nabla w \r\leq \|\nabla w\|\|\u\|_{L^3}\|\phi\|_{L^6}
\leq \Q(R)\|\nabla w\| \|\u\|_{}^{1/2} \|\u\|_{1}^{1/2},
\end{align*}
providing
\begin{align*}
\Big|\int_0^T \l \nabla\cdot (\phi\u),w \r\,\d t\Big|
&\leq \Q(R) \int_0^T \|\nabla w\| \|\u\|_{}^{1/2} \|\u\|_{1}^{1/2} \, \d t\\
&\leq \Q(R) \Big(\int_0^T \|\nabla w\|^2 \,\d t\Big)^{1/2}
\Big(\int_0^T \|\u\| \|\u\|_{1}^{} \,\d t\Big)^{1/2}\\
&\leq \frac{C_T}{\nu^{1/4}}\Big(\int_0^T\|\nabla w\|^2\,\d t\Big)^{1/2}.
\end{align*}
Therefore, if $\nu>0$, then
$$\nabla\cdot (\phi\u)\in L^{2}(0,T; \H^{-1}).$$
\medskip
\noindent
\textbf{The term $\phi\nabla\mu$.} Let $\v\in {\boldsymbol H}$. Thanks to Agmon's inequality,
we infer
\begin{align*}
\l \phi\nabla\mu,\v \r&\leq \|\v\|\|\nabla\mu\|\|\phi\|_{L^\infty}
\leq \|\v\|\|\nabla\mu\|\|\phi\|_1^{1/2}\|\phi\|_2^{1/2}
\leq \Q(R)\|\v\|\|\nabla\mu\|\|\phi\|_2^{1/2}.
\end{align*}
On account of \eqref{E:l4h2}, we can estimate as follows
\begin{align*}
\Big|\int_0^T \l\phi\nabla\mu,\v\r\,\d t\Big|
&\leq \Q(R)\int_0^T\|\v\|\|\nabla\mu\|\|\phi\|_2^{1/2}\,\d t\\
&\leq\Q(R)\Big(\int_0^T\|\v\|^{8/3}\,\d t\Big)^{3/8}
  \Big(\int_0^T\|\nabla \mu\|^2\,\d t\Big)^{1/2}
  \Big(\int_0^T\|\phi\|_2^4\,\d t\Big)^{1/8}\\
&\leq \Q_T(R)\Big(\int_0^T\|\v\|^{8/3}\,\d t\Big)^{3/8},
\end{align*}
which yields, independently of $\nu$,
$$\phi\nabla\mu\in L^{8/5}(0,T;{\boldsymbol H}).$$

\section{Well-posedness for $\nu>0$}
\label{wpCHB}
Aim of this section is proving Theorem \ref{t:wellposed}.
As a matter of fact, due the appearance of regularizing term $-\nu\Delta \u$  in the Brinkman equation, the
(global) existence  can be easily obtained by using a standard Galerkin procedure based on  the formal energy estimates in the previous section.
We refer the reader to \cite{LTZ,WZ} for some details on the procedure; see also
Section~\ref{S:nu_to_0} where the argument needed to pass to the limit in the suitable Galerkin scheme is detailed in a weaker setting.

\smallskip
Instead,  the continuous dependence  estimates \eqref{dipcont} and \eqref{dipcont2} (hence uniqueness) are more delicate and we prove it in some details,  showing the crucial role played by  $\nu>0$.

\subsection{Continuous Dependence and Uniqueness}\label{SS:continuous_dependence}
\noindent
Let $\nu>0$ and $\eta>0$ be fixed, and consider $(\phi_1,\u_1)$ and $(\phi_2,\u_2)$ two weak solutions to the CHB system such that
$\l\phi_1(0)\r=\l\phi_2(0)\r$.
Their difference $\bar \phi=\phi_1-\phi_2$, $\bar\u =\bar u_1-\bar u_2$ solves a.e.\ $t \in [0,T]$
\begin{align}
\label{D1} &\l \partial_t\bar\phi(t), w\r+\l \nabla\cdot(\phi_1(t)\bar\u(t)),w\r+\l \nabla\cdot(\bar\phi(t)\u_2(t)),w\r  \\
&+\l \nabla\bar\mu(t),\nabla w\r=0,\quad \forall w\in \H^1,\nonumber\\
\label{D2}&\nu\l \nabla\bar\u(t),\nabla \v\r+\eta \l \bar\u(t),\v\r  \\
&=-\l \phi_1(t)\nabla \bar\mu(t), \v\r-\l \bar\phi(t)\nabla \mu_2(t), \v\r,\quad \forall\v\in \HD, \nonumber
\end{align}
where
\begin{equation*}\label{Dmu}
    \bar \mu =-\Delta \bar\phi+[f(\phi_1)-f(\phi_2)]
\end{equation*}
and
$\l\bar\phi\r=0.$\\

\noindent Taking $w=-\Delta \bar\phi$ in \eqref{D1}, we get
$$\ddt \frac12 \|\nabla \bar\phi\|^2+\l \phi_1\bar\u,\nabla\Delta \bar\phi\r
+\l \bar\phi\u_2,\nabla\Delta \bar\phi\r-\l \nabla \bar\mu,\nabla\Delta \bar\phi\r=0,$$
with
$$\l \nabla \bar\mu,\nabla\Delta \bar\phi\r=-\|\nabla\Delta \bar\phi\|^2+
\l \nabla[f(\phi_1)-f(\phi_2)],\nabla\Delta \bar\phi\r.$$
Thus we obtain
\begin{equation}
\label{uno}\ddt \frac12 \|\nabla \bar\phi\|^2+\|\nabla\Delta \bar\phi\|^2=-\l \phi_1\bar\u,\nabla\Delta \bar\phi\r -\l \bar\phi\u_2,\nabla\Delta \bar\phi\r+
\l \nabla[f(\phi_1)-f(\phi_2)],\nabla\Delta \bar\phi\r.
\end{equation}
Taking $\v=\bar \u$ in \eqref{D2} yields
$$\nu\|\nabla\bar\u\|^2+\eta\|\bar \u\|^2=-\l \phi_1\nabla\bar\mu,\bar\u\r-\l \bar \phi\nabla\mu_2,\bar\u\r.$$
Note that, by definition of $\bar\mu$, we have
$$-\l \phi_1\nabla\bar\mu,\bar\u\r=\l \phi_1\nabla\Delta\bar\phi,\bar\u\r-
\l \phi_1\nabla[f(\phi_1)-f(\phi_2)],\bar\u\r,$$
so that the terms $\pm\l \phi_1\bar\u,\nabla\Delta \bar\phi\r$ get canceled when adding with \eqref{uno}.
Therefore we end up with
\begin{align*}
&\ddt \frac12 \|\nabla \bar\phi\|^2+\|\nabla\Delta \bar\phi\|^2+
\nu\|\nabla\bar\u\|^2+\eta\|\bar \u\|^2\\
&= -\l \bar\phi\u_2,\nabla\Delta \bar\phi\r - \l \bar \phi\nabla\mu_2,\bar\u\r -\l \phi_1\nabla[f(\phi_1)-f(\phi_2)],\bar\u\r+\l \nabla[f(\phi_1)-f(\phi_2)],\nabla\Delta \bar\phi\r.
\end{align*}
We now estimate the right hand side in light of the energy estimates in Section~\ref{S:basic_estimates}. This, in particular, gives
$$\sup_{t\geq 0}(\|\phi_1(t)\|_1+\|\phi_2(t)\|_1)\leq \Q(R),$$
with $R=\max\{\|\phi_1(0)\|_1,\|\phi_2(0)\|_1\}$.
First of all, we have
\begin{align*}
-\l \bar\phi\u_2,\nabla\Delta \bar\phi\r \leq  \|\bar\phi\|_1 \|\u_{2}\|_{}^{1/2} \|\u_{2}\|_{1}^{1/2} \|\nabla\Delta \bar\phi\|
\leq \frac{1}{4}\|\nabla\Delta \bar\phi\|^2 + \frac{h(t)}{\nu^{1/2}} \|\bar\phi\|_1^2,
\end{align*}
where $h(t):=c \nu^{1/2} \|\u_{2}(t)\| \|\u_{2}(t)\|_{1}$
and $c>0$ is independent of $\nu$.

\smallskip
\noindent
Next, observe that the following estimate holds
\begin{align}\label{unaltro}
-\l \bar \phi\nabla\mu_2,\bar\u\r \leq \|\bar\phi\|_1 \|\bar{\u}\|_{L^{3}} \|\nabla \mu_2\|
\leq \frac{\nu}{2} \|\nabla\bar\u\|^{2} + \frac{\eta}{4} \|\bar{\u}\|^{2} + \frac{k(t)}{\eta^{1/2}\nu^{1/2}} \|\bar\phi\|_1^2,
\end{align}
where $k(t) := c \|\nabla \mu_2(t)\|^2$ for some $c>0$, independent of $\nu$.

\smallskip
\noindent
In order to deal with the term
\begin{align*}
\l \nabla[f(\phi_1)-f(\phi_2)],\nabla\Delta \bar\phi\r&\leq \|\nabla[f(\phi_1)-f(\phi_2)]\|\|\nabla\Delta \bar\phi\|\\&\leq \frac{1}{4}\|\nabla\Delta \bar\phi\|^2+C \|\nabla[f(\phi_1)-f(\phi_2)]\|^2,
\end{align*}
we observe that
$$
\|\nabla[f(\phi_1)-f(\phi_2)]\|^2\leq \|[f'(\phi_1)-f'(\phi_2)]\nabla \phi_1\|^2+ \|f'(\phi_2)\nabla\bar\phi\|^2.
$$
We estimate the latter term on the right hand side in light of \eqref{GROW}, \eqref{e-base} and interpolation, that is,
\begin{align*}
\|f'(\phi_2)\nabla\bar\phi\|^{2}
\leq c\int_\Omega(1+|\phi_2|^4)|\nabla\bar\phi|^2
\leq c(1+\|\phi_2\|^4_{L^\infty})\|\bar\phi\|_{1}^2
\leq \Q(R)(1+\|\phi_2\|^2_2)\|\bar\phi\|_{1}^2,
\end{align*}
and arguing analogously for the former, we get
\begin{align*}
\|[f'(\phi_1)-f'(\phi_2)]\nabla \phi_1 \|^{2}&\leq c\int_\Omega |(1+|\phi_1|+|\phi_2|)\bar\phi\nabla\phi_1|^{2}
\leq \Q(R)\|\bar\phi\|_1^{2}\|\phi_1\|_2^2.
\end{align*}
This proves
\begin{equation}
\label{cff}
\|\nabla[f(\phi_1)-f(\phi_2)]\|^2\leq \ell(t)\|\bar\phi\|_1^{2},
\end{equation}
where
$\ell(t):=\Q(R)(1+\|\phi_1(t)\|_2^{2}+\|\phi_2(t)\|_2^{2})$.
In order to control the remaining term, we exploit \eqref{cff}
in the following way
\begin{align}
\label{pallida}
\l \phi_1\nabla[f(\phi_1)-f(\phi_2)],\bar\u\r&\leq \|\phi_1\|_{L^6} \|\nabla[f(\phi_1)-f(\phi_2)]\| \|\bar \u\|_{L^{3}}\\\nonumber
&\leq
\Q(R)\|\bar{\u}\| ^{1/2}\|\nabla\bar\u\|^{1/2} \|\nabla[f(\phi_1)-f(\phi_2)]\|\\\nonumber
&\leq
\frac{\nu}{2}\|\nabla\bar\u\|^2+\frac\eta2\|\bar{\u}\|^2 + \frac{\ell(t)}{\nu^{1/2}\eta^{1/2}}\|\bar\phi\|_1^{2}.\nonumber
\end{align}
Collecting the above estimates we get
\begin{equation}
\label{olivastra}
\ddt  \|\bar\phi\|_1^2 + \frac{\nu}{2}\|\nabla\bar\u\|^2+\frac\eta2\|\bar{\u}\|^{2} \leq \frac{g(t)}{\nu^{1/2}\eta^{1/2}} \|\bar\phi\|_1^2,
\end{equation}
where
$g(t):=h(t)+k(t)+\ell(t)$, on account of \eqref{e-base} and \eqref{e2-base}, satisfies
$$\int_0^T g(y)\,\d y\leq \Q_T(R).$$
Hence an application of the standard Gronwall lemma gives
$$
\|\phi_1(t)-\phi_2(t)\|_1^2\leq \|\phi_1(0)-\phi_2(0)\|_1^2 \e^{\nu^{-1/2}\eta^{-1/2}\int_0^t g(y)\,\d y},
$$
which proves \eqref{dipcont}.
An integration of \eqref{olivastra} yields the further bound \eqref{dipcont2}.
Finally, letting $\phi_1(0)=\phi_2(0)$ in \eqref{dipcont} and \eqref{dipcont2} we obtain $\phi_1(t)=\phi_2(t)$ and $\u_1(t)=\u_2(t)$
for almost every $t$, i.e., uniqueness.

\medskip
\smallskip
We observe that, when $\eta=0$ (see Remark~\ref{r:dc0}), the only changes needed in the proof of the continuous dependence estimate are in \eqref{unaltro} and in \eqref{pallida}, which now become
\begin{align*}
-\l \bar{\phi} \nabla \mu_{2}, \bar{\u}, \bar{\u} \r
& \leq \frac{\nu}{4} \|\nabla\bar{\u}\|^{2} + \frac{k(t)}{\nu}\|\bar{\phi}\|_{1}^{2}\\\nonumber
-
\l \phi_1\nabla[f(\phi_1)-f(\phi_2)],\bar\u\r
&\leq
\frac{\nu}{2}\|\nabla\bar\u\|^2+\frac{\ell(t)}{\nu}\|\bar\phi\|_1^{2}.
\end{align*}
%%%%%%%%%%%%%%%%%%%%%%%%%%%%%%%%%%%%%

%%%%%%%%%%%%%%%%%%%%%%%%%%%%%%%%%%%%%
\subsection{The semigroup $S_\nu(t)$}
Let $I\in\R$ and consider the subspace of $\H^1$
$$V_I=\{\phi\in \H^1\,:\,\l \phi \r=I\}.$$
An immediate consequence of the results of Section~\ref{SS:continuous_dependence}
is that, for any fixed $\nu>0$, system \eqref{B1}-\eqref{BCfi} generates a semigroup
$$S_\nu(t):V_I\to V_I$$
defined by the rule $S_\nu(t)\phi_0=\phi(t)$,
where
$(\phi,\u)$ is the unique
global (weak) solution to system \eqref{B1}-\eqref{BCfi}.
Furthermore, owing to the continuous dependence estimate \eqref{dipcont}, the semigroup is strongly continuous, namely,
$S_\nu(t)\in \C(V_I,V_I)$.
Notice that  the energy estimates of Section \ref{S:basic_estimates}
yield in particular the boundedness of each trajectory
\begin{equation}
\label{e1}
\|S_\nu(t)\phi_0\|_1=\|\phi(t)\|_1\leq c ,\quad \forall t\geq 0,
\end{equation}
where, from now on, $c\geq 0$ denotes a generic constant
that may depend on $\|\phi_0\|_1$ but is independent of the particular $\phi_0$.

\smallskip
\noindent
Besides, if the nonlinearity $f$ satisfies further dissipativity  assumptions stronger than \eqref{DISS1}, it is possible to prove that the dynamical system
$(V_I,S_\nu(t))$ is \emph{dissipative} for any fixed $I\in\mathbb{R}$ (and $\nu>0$). This means that  there exists a \emph{bounded absorbing set} $\mathcal B\subset V_I$ with
the following property:
for every $R>0$ there exists $t_R>0$ such that
$$S_\nu(t)\phi_0\in {\mathcal B}, \quad \forall t\geq t_R,$$
for every $\phi_0\in V_I$ with $\|\phi_0\|_1\leq R$.
This is witnessed by the following result.
\begin{proposition}
\label{p:abs0}
Let the assumptions of Theorem~\ref{t:wellposed} hold and let us assume that
for some $c_0\geq 0$, $c_i>0$, $i=1,2$ and $q>2$ there hold
$$
f(s)s\geq c_1F(s) -c_0\quad\text{and}\quad
F(s)\geq c_2|s|^q-c_0,
$$
for all $s\in\mathbb{R}$.
Then,
\begin{equation}
\label{absorbing}
\|S_\nu(t)\phi_0\|_1\leq \Q(\|\phi_0\|_1)\e^{-kt/2}+R_I,\quad \forall t\geq 0,
\end{equation}
for some $k>0$, where $R_I>0$  depends on $I=\l\phi_0\r$ but  is independent of $\phi_0$.
\end{proposition}
The proof is standard and it is therefore omitted.
%%%%%%%%%%%%%%%%%%%%%%%%%%%%%%%%

%%%%%%%%%%%%%%%%%%%%%%%%%%%%%%%%
\section{Higher order estimates}\label{S:Higher_order}
Here we proceed formally relying on the Galerkin approximation scheme introduced in the previous section.
For the sake of simplicity, from now on we set $\eta=1$ (see Remark \ref{r:eta0}, however).
\begin{proposition}
\label{t:pH2}
Let the assumptions of Theorem~\ref{t:wellposed} hold and suppose, in addition, $f\in \C^2(\R)$ satisfying \eqref{GROW3}.
Then the following estimate holds
\begin{equation}
\label{puntualeH2}
\|\phi(t)\|_2+\int_t^{t+1}\|\phi(y)\|_4^2\,\d y \leq c \left( 1+\frac{1}{\nu}\right) ,\quad \forall t\geq 1.
\end{equation}
\end{proposition}

\begin{proof}
Taking $\v=\u$ in equation \eqref{B2-w} we get
\begin{equation}
\label{Stokes-id}
\nu \| \nabla \u \|^2+\|\u\|^2 = \l \mu\nabla\phi,\u\r.
\end{equation}
By \eqref{GROW} and \eqref{e1} we have
$$\|\mu\|\leq \|\Delta \phi\|+\|f(\phi)\|\leq c(1+\|\Delta \phi\|).$$
Thus, for $\nu>0$, we can estimate the latter term as follows
\begin{align*}
    &\l \mu \nabla\phi,\u\r
    \leq  \| \mu \| \| \nabla \phi \|_{L^{6}}\| \u \|_{L^{3}}
    \leq c (1 + \|\Delta\phi\|) \|\Delta\phi\|\|\u\|^{1/2} \|\nabla\u\|^{1/2} \\
    &\quad\leq \frac{1}{2} \|\u\|^{2} + \frac{\nu}{2} \|\nabla\u\|^{2} + \frac{c}{\nu^{1/2}} (1 + \|\Delta\phi\|^{2}) \|\Delta\phi\|^{2}.
\end{align*}
From this we deduce
\begin{equation}\label{stimu1}
\nu \|\nabla\u\|^{2}+\|\u\|^{2} \leq \frac{c}{\nu^{1/2}} (1 + \|\Delta\phi\|^{2}) \|\Delta\phi\|^{2}.
\end{equation}
Let us now take $w= \Delta^2\phi$ in equation \eqref{B1-w}. This yields
$$\frac12\ddt \|\Delta\phi\|^2+ \|\Delta^2 \phi\|^2= \l\Delta f(\phi),\Delta^2\phi\r+\l\u\nabla\phi,\Delta^2\phi\r.$$
On the other hand, we have
$$ \l\Delta f(\phi),\Delta^2\phi\r\leq \frac14  \|\Delta^2 \phi\|^2+ c\|\Delta f(u)\|^2,$$
where $\|\Delta f(\phi)\|^2$ can be controlled in the following way. Observe that
$$\Delta f(\phi)=\nabla(f'(\phi)\nabla\phi)=f''(\phi)|\nabla\phi|^2 +f'(\phi)\Delta\phi.$$
Then, using \eqref{GROW3}, by the Agmon inequality we get
\begin{align*}
    \|f''(\phi) |\nabla\phi|^{2}\| &\leq c (1 + \|\phi\|_{L^{\infty}}) \|\nabla\phi\|_{L^{4}}^{2} \leq c (1 + \|\Delta\phi\|^{1/2}) \|\Delta\phi\|^{3/2},\\
    \|f'(\phi)\Delta\phi\| &\leq c(1 + \|\phi\|_{L^\infty}^2) \|\Delta\phi\| \leq c (1 + \|\Delta\phi\|) \|\Delta\phi\|.
\end{align*}
Therefore, we obtain
\begin{equation}
\label{Delta_f}
\|\Delta f(\phi)\|^2\leq c(1+\|\Delta\phi\|^2) \|\Delta\phi\|^{3}.
\end{equation}
In order to deal with the remaining term, exploiting \eqref{stimu1} we find
\begin{align*}
    &\l\u\nabla\phi,\Delta^2\phi\r
    \leq c\|\u\|_{L^{3}} \|\nabla\phi\|_{L^{6}} \|\Delta^2\phi\|
    \leq \|\u\|^{1/2} \|\nabla\u\|^{1/2} \|\Delta\phi\| \|\Delta^{2}\phi\|\\
   &\quad\leq \frac{c}{\nu^{1/2}} (1 + \|\Delta\phi\|) \|\Delta\phi\|^{2} \|\Delta^{2}\phi\|
    \leq \frac14 \|\Delta^2 \phi\|^2+\frac{c}{\nu} (1 + \|\Delta\phi\|^2) \|\Delta\phi\|^{4}.
\end{align*}
We thus end up with the differential inequality
$$\frac12\ddt \|\Delta\phi\|^2+\frac12  \|\Delta^2 \phi\|^2 \leq c(1 + \|\Delta\phi\|^{2}) \|\Delta\phi\|^3 + \frac{c}{\nu} (1 + \|\Delta\phi\|^{2}) \|\Delta\phi\|^{4}.$$
Recalling that $\phi\in L^4(0,T;\H^2)$ (see~\eqref{E:l4h2}), Lemma \ref{gen-gronw} yields the claimed result.
\end{proof}

\begin{remark}
\label{soluzioneforte}
Estimate \eqref{puntualeH2} entails that the weak solution $\phi$ is indeed a strong one for $t\geq 1$ (see \eqref{B1}).
In addition, if $\Omega$ is of class $C^{1,1}$, then the regularity of $\mu\nabla\phi$ implies that the weak solution $\u$ to \eqref{B2-w} belongs
to $L^2_{loc}((1,\infty;(H^2(\Omega))^3)$ and the pressure $p$ (see Remark~\ref{pressure}), unique up to an additive function of $t$ only,
belongs to $L^2_{loc}((1,\infty);H^1(\Omega))$ (see, e.g., \cite[Theorem ~IV.5.8]{BoyerFabrie2013}).
Thus equation \eqref{B2} is also satisfied almost everywhere if $t\geq 1$.
\end{remark}

We conclude this section by proving the existence of the global attractor, namely,

\begin{theorem}
\label{globattr}
Let $f$ satisfy all the assumptions in Proposition \ref{p:abs0} and Proposition \ref{t:pH2}. Then the dynamical system
$(V_I, S_\nu(t))$ possesses a (unique) global attractor $\mathcal A$ which is bounded in $\H^2$.
\end{theorem}

\begin{proof}
On account of the assumptions on $f$, thanks to Proposition \ref{p:abs0} and Proposition \ref{t:pH2}, we infer the existence of a \emph{compact} absorbing set (bounded in $\H^2$) for the semigroup $S_\nu(t)$. Hence, by standard results (see, e.g., \cite{Temam1997}) the proof follows.
\end{proof}

\begin{remark}
\label{r:attractor}
We recall that the global attractor $\mathcal A$ is the smallest (for the inclusion)
compact set of the phase space which is invariant by the flow (i.e., $S_\nu(t)\mathcal A = \mathcal A$, $\forall t\geq 0$) and
attracts all bounded sets of initial data as time goes to infinity, namely,
$$\forall B\subset V_I \text{ bounded, }\quad \text{\it dist}(S_\nu(t)B,\mathcal A)\to 0\quad\text{ as }t\to\infty,$$
where {\it dist} denotes the Hausdorff semi-distance between sets in $\H^1$.
\end{remark}
%%%%%%%%%%%%%%%%%%%

%%%%%%%%%%%%%%%%%%%
\section{Convergence to equilibria}
\label{convequil}
In what follows we let $\nu>0$ be fixed omitting in the notation the dependence on $\nu$.
Aim of this section is discussing Theorem \ref{LS}, showing in particular
that,  for every fixed $\phi_0\in \H^1$,
the global solution $\phi(\cdot)=S(\cdot)\phi_0$ originating from $\phi_0$ converges to an equilibrium $\phi^\star$ as $t\to\infty$,  with a certain convergence rate.

To this aim, we first recall that the $\omega$-limit set of $\phi_0$ is defined as
$$\omega(\phi_0)=\{\phi^\star\in \H^1\,:\, \text{$\phi(t_n) \to \phi^\star$ in $\H^1$,\; for some $\{ t_{n} \}_{n \in \mathbb{N}}$, $t_n\to\infty$}\},$$
and that the set of stationary points associated with $\phi_0$ is
$$\S(\phi_0)=\{z\in\H^2\,:\, -\Delta z+ f(z)=const,\, \l z\r=\l\phi_0\r\}.$$
Recall that by \eqref{puntualeH2} in the form
\begin{equation}
\label{rc}
\| \phi(t) \|_{2} \leq c,\quad t\geq 1,
\end{equation}
where along  the section $c>0$ denotes a generic constant, possibly depending on $\|\phi_0\|_1$ (and increasing as  $\nu\to 0$).
Hence, due to to the compact embedding $\H^2\hookrightarrow \H^1$,
the $\omega$-limit set of $\phi_0\in \H^1$ is a nonempty compact subset of $\H^1$.

\smallskip

With this notation,  the main step in the proof of Theorem \ref{LS}  consists in showing that each $\omega$-limit set
consists in one \emph{single} stationary state, as stated in the following
\begin{proposition}
\label{p:ce}
There exists $\phi^\star\in \S(\phi_0)$ such that $\omega(\phi_0)=\{\phi^\star\}$.
\end{proposition}

Notice that, since $\omega(\phi_0)\neq \emptyset$, there exists  $\phi^\star\in\H^1$ and $t_n\to\infty$ such that
\begin{equation}
\label{lsucc}
\|\phi(t_n)-\phi^\star\|_1\to 0,\quad \mbox{ as } n\to\infty,
\end{equation}
Besides, owing to due to \eqref{s-base}, it is easy to realize that
the functional
$$E(z)=\frac12 \|\nabla z\|^2+\l F(z),1\r$$
with $z\in \H^1$ is a Lyapunov functional for $S(t)$. Thus, by standard results on gradient systems
(see, e.g., \cite[Chapter~9]{Cazenave1998}), we learn that $\omega(\phi_0)\subset\S(\phi_0)$ proving in particular that
$\phi^\star\in \S(\phi_0)$.
As a consequence, the proof of Proposition \ref{p:ce} will follow by showing that the \emph{whole trajectory} $\phi(\cdot)$ converges to $\phi^\star$, namely
\begin{equation}
\label{ltot}\lim_{t\to\infty}\|\phi(t)-\phi^\star\|_1=0.
\end{equation}

The proof of this fact can be obtained by a well-known contradiction argument due to \cite{Jendoubi1998} (see also \cite{ ZhaoWuHuang2009} and references therein), known as {\L}ojasiewicz-Simon approach.
We omit it, since it can be obtained by reasoning as in \cite[Section 3.3]{WW} with minor changes.
Let us only mention that
the key tool in our situation
is the following version of the {\L}ojasiewicz-Simon inequality (see  \cite[Proposition~6.3]{Abels2007}).

\begin{theorem}
\label{LSineq}
Let $(\phi, \u)$ be a solution of system~\eqref{B1}-\eqref{Bdiv} with initial datum $\phi_{0}\in H^1$ and let $z\in \omega(\phi_{0}) \subset \S(\phi_{0})$.
If $f$ is real analytic and satisfies \eqref{DISS0}, then there exist $\theta=\theta(z)>0$, $\theta \in (0, \tfrac{1}{2})$ and
$\varsigma=\varsigma(z)>0$ such that
\begin{equation}
\label{SL}
|E(\phi) - E(z)|^{1-\theta} \leq
\|\mathbf{P}(\Delta \phi-f(\phi))\|_{\H^{-1}},
\end{equation}
whenever $\phi$ fulfills
$\|\phi-z\|_1<\varsigma$.
Here $\mathbf{P}:\H\to \H$ is defined by $\mathbf{P}(u)=u-\l u\r$.
\end{theorem}

\smallskip

 The next step consists in obtaining the rate of convergence of the trajectory to the equilibrium.
This is witnessed by the following:

\begin{proposition}
\label{CR}
Let $\theta=\theta(\phi^\star)\in (0,\tfrac{1}{2})$ as provided by Theorem \ref{LSineq}. Then,
\begin{equation}
\label{fiest}
\|\phi(t)-\phi^\star\|_1 \leq \frac{c}{(1+t)^{\theta/(1-2\theta)}},
\end{equation}
for some $c=c(\phi_0)\geq 0$ and every
$t\geq t^\star$, for some $t^\star>0$.
\end{proposition}

\begin{proof}
Reasoning as in \cite[Section 3.4]{WW}  (cf.~also \cite[5.2]{ZhaoWuHuang2009}),
it is easy to prove that \eqref{fiest} holds for the weaker norm $\|\phi(t)-\phi^\star\|_{\H^{-1}}$, namely
\begin{equation}
\label{st-1}
\|\phi(t)-\phi^\star\|_{\H^{-1}}+\int_{t}^{\infty}\|\nabla\mu(y)\|\,\d y\leq
\frac{c}{(1+t)^{\theta/(1-2\theta)}},\quad t\geq t^\star.
\end{equation}
In order to complete the proof,
we set
$$\Phi(t)=\phi(t)-\phi^\star,$$
and observe that, for $t\geq t^\star$ and almost everywhere in $\Omega$ (cf. Remark~\ref{soluzioneforte}), there holds
\begin{equation}
\label{eqd}
\partial_t \Phi +\u\cdot\nabla(\Phi+\phi^\star)=\Delta\big(-\Delta \Phi+[f(\phi)-f(\phi^\star)]\big).
\end{equation}

\noindent Recalling \eqref{Stokes-id}, we have
\begin{align*}
&\nu \|\nabla\u\|^2 + \|\u\|^2=\l \phi \nabla \mu, \u \r
\leq \|\nabla \mu\| \|\u\| \|\phi\|_{L^{\infty}}\\
&\quad\leq \|\nabla \mu\| \|\u\| \|\phi\|_{1}^{1/2} \|\phi\|_{2}^{1/2}
\leq \frac{1}{2} \|\u\|^{2} + c \|\nabla \mu\|^{2}.
\end{align*}
Besides,  since $\phi^\star\in \S(\phi_0)$, then $\mu^{\star}:=- \Delta \phi^\star + f(\phi^\star)$ is constant,
and we can estimate
\begin{align*}
&\|\nabla \mu\|=\|\nabla(\mu-\mu^\star)\| \leq \|\nabla \Delta \Phi\| + \|\nabla (f(\phi) - f(\phi^{\star}))\|
\leq \|\nabla \Delta \Phi\|+c\|\nabla\Phi\|\\&\quad\leq \|\nabla \Delta \Phi\|+c\|\Phi\|_{\H^{-1}}^{1/2}\|\nabla \Delta \Phi\|^{1/2}
 \leq c\|\nabla \Delta \Phi\| + \|\Phi\|_{\H^{-1}}.
\end{align*}
In particular, we obtain
\begin{equation}
\label{uuu}
\nu \|\nabla\u\|^2 +  \frac12\|\u\|^2 \leq c  \| \nabla \mu \|^2 \leq c\|\nabla \Delta \Phi\|^2 + c\|\Phi\|_{\H^{-1}}^2.
\end{equation}
Taking the product of~\eqref{eqd} with $-\Delta \Phi$ in $\H$ we have
$$\frac12 \ddt\|\nabla \Phi\|^2 + \|\nabla \Delta \Phi\|^2 =
-  \l \Delta[f(\phi)-f(\phi^\star)], \Delta \Phi \r+\l \u\cdot\nabla(\Phi+\phi^\star), \Delta \Phi \r.$$
Observe that
$$\l \Delta[f(\phi)-f(\phi^\star)], -\Delta \Phi \r \leq \|\nabla[f(\phi)-f(\phi^\star)]\| \| \nabla \Delta \Phi \| \leq c \| \nabla \Phi \|^2 + \frac14 \| \nabla \Delta \Phi \|^2.$$
The latter term can be estimated in light of \eqref{uuu} as
\begin{align*}
&\l \u \cdot \nabla (\Phi+\phi^\star), -\Delta \Phi\r
\leq \| \nabla\phi \|_{L^{6}} \| \u \| \| \Delta \Phi \|_{L^{3}}
\leq c  \| \u\|\| \Delta \Phi \|^{1/2} \| \nabla \Delta \Phi \|^{1/2}\\
&\quad\leq  c(\|\nabla \Delta \Phi\| + \|\Phi\|_{\H^{-1}}) \| \nabla \Delta \Phi \|^{7/8} \| \Phi \|_{\H^{-1}}^{1/8}
\leq \frac14 \| \nabla \Delta \Phi \|^{2} + c \| \Phi \|_{\H^{-1}}^{2}.
\end{align*}
We thus obtain, on account of \eqref{st-1}, the inequality
$$\frac12 \ddt \| \nabla \Phi \|^2 + \frac12 \| \nabla \Delta \Phi \|^2 \leq c \| \Phi \|_{\H^{-1}}^2   \leq \frac{c}{(1+t)^{2\theta/(1-2\theta)}}.$$
Recalling that $\|\nabla \Phi(t^\star)\|\leq c$, an application of Lemma \ref{gr}
yields \eqref{fiest}.
\end{proof}

\subsection*{Convergence of the velocity field $\u$}
In order to complete the proof of Theorem \ref{LS} we are left to show that
$\|\u(t)\|_1$ decays to $0$ as $t\to\infty$. This requires a different argument than in \cite{WW}, so we detail the proof.
First we need  a further regularization of $\phi$, namely,
\begin{lemma}
The following inequality holds
$$\|\nabla\mu(t)\|+\|\nabla \Delta\phi(t)\|\leq c,\quad\forall t\geq 2.$$
\end{lemma}
\begin{proof}
Taking $w=\Delta^2\mu$ in \eqref{B1-w} we have
\begin{align*}
\l \phi_t,\Delta^2\mu\r=\l \u\nabla\phi+\Delta \mu,\Delta^2\mu\r= -\l\nabla(\u\nabla\phi),\nabla\Delta\mu\r-\|\nabla\Delta\mu\|^2.
\end{align*}
Exploiting the definition of $\mu$,
which gives $\mu_t=-\Delta\phi_t+ f'(\phi)\phi_t$,
we obtain
\begin{align*}
\l \phi_t,\Delta^2\mu\r&=\l -\Delta\phi_t, -\Delta\mu\r
=\l \mu_t, -\Delta\mu\r-\l f'(\phi)\phi_t, -\Delta\mu\r\\
&=\frac12\ddt\|\nabla\mu\|^2-\l f'(\phi)\Delta\mu, -\Delta\mu\r+\l f'(\phi)(\u\nabla\phi), -\Delta\mu\r.
\end{align*}
Hence we deduce that
$$\frac12\ddt\|\nabla\mu\|^2+\|\nabla\Delta\mu\|^2=-\l f'(\phi)\Delta\mu, \Delta\mu\r+\l f'(\phi)(\u\nabla\phi), \Delta\mu\r-\l\nabla(\u\nabla\phi),\nabla\Delta\mu\r.$$
Let us estimate the terms on the right hand side.
Observe that, in light of \eqref{rc}, we have
$$\|f'(\phi)\|_{L^\infty}\leq c.$$
Thus we get
\begin{align*}
-\l f'(\phi)\Delta\mu, \Delta\mu\r&\leq \|f'(\phi)\|_{L^\infty}\|\Delta\mu\|^2\leq c\|\Delta\mu\|^2\\
&\leq c\|\nabla\Delta\mu\|\|\nabla\mu\|\leq\frac14\|\nabla\Delta\mu\|^2+ c\|\nabla\mu\|^2,
\end{align*}
and
\begin{align*}
\l f'(\phi)(\u\nabla\phi), \Delta\mu\r&\leq \|f'(\phi)\|_{L^\infty}\|\Delta\mu\|\|\u\|_{L^4}\|\nabla\phi\|_{L^4}\\
&\leq\|f'(\phi)\|_{L^\infty}\|\nabla \mu\|^{1/2}\|\nabla\Delta\mu\|^{1/2}\|\u\|^{1/4}\|\nabla \u\|^{3/4}
\|\nabla\phi\|^{1/4}\|\nabla^2\phi\|^{3/4}\\
&\leq \frac14\|\nabla\Delta\mu\|^2+ c\|\nabla\mu\|^2+c\|\nabla\u\|^2.
\end{align*}
Furthermore, by the Agmon inequality, we get
\begin{align*}
&\l\nabla(\u\nabla\phi),\nabla\Delta\mu\r\leq\|\nabla\Delta\mu\|(\|\nabla\u\|\|\nabla\phi\|_{L^\infty}
+\|\u\|_{L^6}\|\nabla^2\phi\|_{L^3})\\
&\quad\leq c\|\nabla\Delta\mu\|\|\nabla\u\|\|\nabla\Delta\phi\|^{1/2}
\leq\frac14\|\nabla\Delta\mu\|^2+c\|\nabla \u\|^2\|\nabla\Delta\phi\|.
\end{align*}
Note that $\nabla\mu=-\nabla\Delta\phi+f'(\phi)\nabla\phi$. Then we infer
\begin{equation}
\label{tree}
\|\nabla\Delta\phi\|\leq \|\nabla\mu\|+\|f'(\phi)\|_{L^\infty}\|\nabla\phi\|\leq (\|\nabla\mu\|+c),
\end{equation}
so that
\begin{align*}
\l\nabla(\u\nabla\phi),\nabla\Delta\mu\r
\leq\frac14\|\nabla\Delta\mu\|^2+c\|\nabla \u\|^2+c\|\nabla \u\|^2\|\nabla\mu\|^2.
\end{align*}
We thus end up with the differential inequality
$$\frac12\ddt\|\nabla\mu\|^2+\frac14\|\nabla\Delta\mu\|^2
\leq c(1+\|\nabla \u\|^2+\|\nabla\mu\|^2+\|\nabla \u\|^2\|\nabla\mu\|^2).$$
Recalling that  $\int_0^\infty (\|\nabla \u(y)\|^2+\|\nabla\mu\|^2)\,\d y<c$ (see \eqref{e-base}),
Lemma \ref{gen-gronw} yields
$$\|\nabla\mu(t+1)\|^2\leq c,\quad t\geq 1.$$
Finally, by \eqref{tree}
we also have
$$\|\nabla\Delta\phi(t+1)\|^2\leq c,\quad t\geq 1,$$
as claimed.
\end{proof}
We are now ready to prove the convergence of $\u$ to zero.
To this aim, let us observe that, since $\mu^{\star}$  is constant, then the equation for the velocity field $\u^\star$ associated with $\phi^\star$ reduces to
\begin{equation*}
    \begin{cases}
        -\nu \Delta \u^{\star} + \u^{\star} = -\nabla p^{\star}\\
        \nabla \cdot \u^{\star} = 0.
    \end{cases}
\end{equation*}
Therefore, upon multiplication by $\u^{\star}$ and integration over $\Omega$, we deduce
$$ \nu \| \nabla \u^{\star} \|^{2} + \| \u^{\star} \|^2 = 0 .$$
This implies $\u^{\star} \equiv \mathbf{0}$, and  the following equation for the pressure~$p^{\star}$ holds
$$ \nabla p^\star = \mu^\star\nabla\phi^\star-\nabla(\phi^\star\mu^\star).$$
Subtracting this last equation from~\eqref{B2-w} we deduce
$$-\nu\Delta\u+\u=-\nabla (p-p^\star)+(\mu-\mu^\star)\nabla(\Phi+\phi^\star)+\mu^\star\nabla \Phi+\nabla(\phi^\star\mu^\star).$$
Testing this relation by $\u$,  we obtain
\begin{align*}
\nu \|\nabla\u\|^2 + \|\u\|^2&=\l(\mu-\mu^\star)\nabla(\Phi +\phi^\star),\u\r+\l\mu^\star\nabla \Phi,\u\r\\
&\leq \|\mu-\mu^\star\|_{L^{3}} \|\nabla(\Phi+\phi^\star)\|_{L^6} \|\u\| + |\mu^\star| \|\nabla \Phi\| \|\u\|\\
&\leq \| \mu -\mu^{\star} \|^{1/2} \| \nabla (\mu - \mu^{\star}) \|^{1/2} \| \phi \|_{2} \| \u \| + c \| \nabla \Phi \| \| \u \| \\
&\leq \frac{1}{2} \|\u\|^{2 }+c \| \Delta \Phi \|^{2} +c\| \mu -\mu^{\star} \| \| \nabla \mu  \|.
\end{align*}
Since $
\| \mu - \mu^{\star} \| \leq \| \Delta \Phi \|+\|f(\phi)-f(\phi^\star)\|\leq c\|\Delta \Phi\|,
$
we have  the estimate
\begin{equation*}
\nu \|\nabla\u\|^2 +\frac12 \|\u\|^2\leq c\|\Delta \Phi\|(\| \Delta \Phi \|+\|\nabla \mu\|),
\end{equation*}
and, by exploiting the boundedness of $\nabla \mu$ and $\Delta \Phi$, this yields
$$\|\u\|^2_1\leq c_\nu\|\Delta\Phi\|,$$
for every $t\geq 2$.
Finally, by interpolation and invoking the boundedness of $\|\nabla\Delta\phi\|$,
we have
$$\|\u\|^2_1\leq c_\nu\|\nabla \Phi\|^{1/2}\|\nabla\Delta \Phi\|^{1/2}\leq c_\nu\|\nabla \Phi\|^{1/2}=c_\nu\|\nabla(\phi-\phi^\star)\|^{1/2}.$$
Therefore, Proposition \ref{CR} entails \eqref{convrate-u}.

\begin{remark}
\label{r:eta0}
All the results and the estimates performed in this section and in Section \ref{S:Higher_order} can be carried out in the case $\eta=0$ with minor changes.
\end{remark}
%%%%%%%%%%%%%%%%%

%%%%%%%%%%%%%%%%%%
\section{The limit $\nu\to 0$}
\label{S:nu_to_0}

Before studying the convergence of solutions to CHB system as $\nu \to 0$, we recall the following compactness result (see, e.g., \cite{lions}).
\begin{theorem}
\label{lions}
Let $X_0\subset\subset X\subset X_1$ be three reflexive Banach spaces. Let $1<a,b<\infty$ and define
$$W^{a,b}(0,T; X_0,X_1)=\{z\in L^{a}(0,T;X_0)\,:\,\partial_t z\in L^{b}(0,T;X_1)\}.$$
Then $W^{a,b}(0,T; X_0,X_1)$ is reflexive and
$$W^{a,b}(0,T; X_0,X_1)\hookrightarrow L^{a}(0,T; X)$$
with compact embedding.
\end{theorem}

\subsection{Proof of Theorem \ref{t:nuto0}}
Let $\phi_0\in \H^1$ and let $\{\nu_n\}_{n\in\N}$ be a sequence of positive numbers such that $\nu_n\to 0$ as $n\to\infty$.
Consider the sequence $(\phi_{\nu_n},\u_{\nu_n})$ of weak solutions corresponding to the CHB system with $\nu=\nu_n$.
From the previous sections we know that the following bounds on $\{ \phi_{\nu_n} \}_{n\in N}$, $\{ \u_{\nu_n} \}_{n\in\N}$ and
$\{ \mu_{\nu_n} \}_{n\in\N}$ are independent of $n$:
    \begin{align*}
        &\|\phi_{\nu_n}\|_{L^\infty(0,T;\H^1)}+\|\phi_{\nu_n}\|_{L^2(0,T;\H^3)}\leq c,\\
        &\| \mu_{\nu_n} \|_{L^{2}(0,T;\H^1)} \leq c,\\
        &\|\nabla\cdot(\phi_{\nu_n} \u_{\nu_n})\|_{L^{8/5}(0,T;\H^{-1})}+\|\partial_t\phi_{\nu_n}\|_{L^{8/5}(0,T;\H^{-1})}\leq c,\\
        &\|\u_{\nu_n}\|_{L^{2}(0,T;{\boldsymbol H})}\leq c,\\
        &\|\phi_{\nu_n}\nabla \mu_{\nu_n}\|_{L^{8/5}(0,T;{\boldsymbol H})}\leq c.
    \end{align*}
    Thus we deduce that there exists a relabeled sequence $\{ \nu_{n} \}_{n \in \mathbb{N}}$ such that
    \begin{align*}
    &\phi_{\nu_{n}} \to \phi \quad \text{ weakly in } L^2(0,T;\H^3),\\
    &\mu_{\nu_{n}} \to z \quad \text{ weakly in } L^2(0,T;\H^1),\\
    &\u_{\nu_{n}} \to \u \quad \text{ weakly in } L^2(0,T;{\boldsymbol H}).
    \end{align*}
    By the boundedness of $\partial_{t} \phi_{\nu}$ in $L^{8/5}(0,T;\H^{-1})$ and by the uniqueness of $L^{p}$ and distributional limits, we also have
    $$\partial_t \phi_{\nu_{n}} \to \partial_t\phi\quad \text{ weakly in } L^{8/5}(0,T;\H^{-1}).$$

    \smallskip
    \noindent Applying Theorem~\ref{lions} to $\phi_{\nu_{n}}$ with $X_1=\H^{-1}$ and $X_0=\H^3$, up to a further subsequence, which will be relabeled  $\nu_{n}$, one has
    $$\phi_{\nu_{n}} \to \phi\quad \text{ strongly in } L^{2}(0,T;\H^s),$$
    for all $0\leq s<3$ and
    \begin{equation*}
        \phi_{\nu_{n}} \to\phi\quad \text{ a.e. in } \Omega\times (0,T).
    \end{equation*}
Moreover,  from the regularity of the potential $f$, it follows that $z = - \Delta \phi +f(\phi) =\mu$.

    \smallskip
    \noindent We can now consider the nonlinear terms appearing in~\eqref{B1-w} and \eqref{B2-w}. Let $h$ be a positive real number. First of all, we show convergence of $\phi_{\nu_{n}} \nabla \mu_{\nu_{n}}$ to $\phi \nabla \mu$ in the following (weak) sense
    \begin{align*}
        \int_{t}^{t+h} \l \phi_{\nu_{n}} \nabla \mu_{\nu_{n}} - \phi \nabla \mu, \v \r \, \d t \to 0,\qquad \forall \v\in {\boldsymbol V}.
    \end{align*}
    The integrand can be rewritten as
    $$\l (\phi_{\nu_{n}} - \phi) \nabla \mu_{\nu_{n}}, \v \r + \l \phi[\nabla \mu_{\nu_{n}} - \nabla \mu], \v \r.$$
    The first term in this expression is bounded by
    $$\l (\phi_{\nu_{n}} - \phi) \nabla \mu_{\nu_{n}}, \v \r \leq \|\phi_{\nu_{n}} - \phi\|_{L^3} \|\nabla \mu_{\nu_{n}}\|\|\v\|_{L^6},$$
    so that
    $$\Big|\int_{t}^{t+h} \l (\phi_{\nu_{n}} - \phi) \nabla \mu_{\nu_{n}}, \v \r\,\d t \Big|
    \leq \|\phi_{\nu_{n}} - \phi\|_{L^2(0,T;L^3)} \| \mu_{\nu_{n}} \|_{L^{2}(0,T;\H^1)} \| \v \|_{\boldsymbol V} \to 0.$$
   Recalling that $\phi \in L^{2}(0,T;L^{\infty}(\Omega))$ the weak convergence of $\mu_{\nu_{n}}$ in $L^{2}(0,T;\H^1)$ implies
    $$\l \phi[\nabla \mu_{\nu_{n}} - \nabla \mu], \v \r \to 0.$$

    \smallskip
    \noindent Similarly we can deal with the convergence in $\nabla\cdot (\phi_{\nu_{n}} \u_{\nu_{n}})$. Indeed, we have
    \begin{align*}
        \int_{t}^{t+h} \l \phi_{\nu_{n}} \u_{\nu_{n}} - \phi\u, \nabla v \r \, \d t \to 0,\qquad \forall v \in \H^1.
    \end{align*}
    This can be easily seen by rewriting the integrand as
    $$\l (\phi_{\nu_{n}} - \phi) \u_{\nu_{n}}, \nabla v \r + \l \phi[\u_{\nu_{n}} - \u], \nabla v \r.$$
    Indeed, the second term vanishes as $n\to\infty$ in light of the convergence
    $$\u_{\nu_{n}} \to \u \quad \text{ weakly in } L^2(0,T;{\boldsymbol H})$$
    and recalling the bound $\phi\in L^2(0,T;\H^2)\subset L^{2}(0,T;L^{\infty}(\Omega))$, which yields $\phi\nabla v\in L^2(0,T;(\H)^3)$. Concerning the former, we observe
    \begin{align*}
    \Big|\int_{t}^{t+h} \l [\phi_{\nu_{n}} - \phi]\u_{\nu_{n}}, \nabla v \r \, \d t\Big|&\leq
    \int_{t}^{t+h} \|\nabla v\| \| \u_{\nu_{n}}\| \|\phi_{\nu_{n}} - \phi\|_{L^\infty}\d t\\
    &\leq \| \nabla v \| \Big(\int_{t}^{t+h} \| \u_{\nu_{n}}\|^2 \,\d t\Big)^{1/2}
    \Big(\int_{t}^{t+h} \|\phi_{\nu_{n}} - \phi\|_{L^\infty}^2\,\d t\Big)^{1/2}.
    \end{align*}
    An application of  Theorem~\ref{lions} yields the compactness of  $\{ \phi_{\nu_{n}} \}$  in $L^{2}(0,T;L^{\infty}(\Omega))$, proving the required convergence.

        \smallskip
    \noindent Finally, let us consider the term involving the time derivative of $\phi$. In particular, recalling that $v$ is constant in time, we have
    $$ \int_{t}^{t+h} \partial_{t} \phi \, v \, \d t = (\phi(t+h) - \phi(t))v. $$
    Thanks to the boundedness of $\partial_{t} \phi$ in $L^{8/5}(0,T;\H^{-1})$, the Lebesgue Theorem also gives
    $$ \frac{\phi(t+h) - \phi(t)}{h} \to \partial_{t} \phi(t) \quad \text{ a.e.\ $t \in [0,T]$}.$$
    A repeated application of the Lebesgue Theorem implies that the couple $(\phi, \u)$ satisfies~\eqref{B1-w}-\eqref{B2-w} for almost every time $t \in [0,T]$.
    Moreover, observing that $\phi$ in $L^{\infty}(0,T;\H^1)$ and $\phi\in \C([0,T];\H^{-1})$, it follows that $\phi$ is also weakly continuous taking values in $\H^1$.

    \smallskip
    \noindent Finally we show that
    $$ \lim_{t \to 0} \l \phi(t), v \r = \l \phi_{0}, v \r, \quad \text{for all $v \in \H^{-1}$}.$$
    Let $\psi \colon [0,T] \to \mathbb{R}$ be a $\C^{\infty}$ function such that $\psi(0) = 1$ and $\psi(T) = 0$ and let $v \in \H^1$ be arbitrary. Multiplying~\eqref{B1-w} with $\nu > 0$ by $\psi v$ and integrating over $\Omega \times [0,T]$ we obtain
    $$-\int_{0}^{T} \l \phi_{\nu_{n}}, \psi v \r \, \d t + \int_{0}^{T} \l \phi_{\nu_{n}} \u_{\nu_{n}}, \psi \nabla v \r \d t+ \int_{0}^{T} \l \nabla \mu_{\nu_{n}}, \psi \nabla v \r \d t = \l \phi_{0}, v \r.$$
    As before, we can pass to the limit as $\nu_{n} \to 0$, so obtaining
    $$-\int_{0}^{T} \l \phi, \psi v \r \, \d t + \int_{0}^{T} \l \phi \u, \psi \nabla v \r \d t + \int_{0}^{T} \l \nabla \mu, \psi \nabla v \r \d t = \l \phi_{0}, v \r.$$
    Proceeding analogously  in the case $\nu = 0$, we deduce
    $$-\int_{0}^{T} \l \phi, \psi v \r \, \d t + \int_{0}^{T} \l \phi \u, \psi \nabla v \r \d t + \int_{0}^{T} \l \nabla \mu, \psi \nabla v \r \d t = \l \phi(0), v \r.$$
    Finally, a comparison between these last two equalities and the arbitrary choice of $v \in \H^1$ gives
    $\phi(0) = \phi_{0}.$ \qedhere
%%%%%%%%%%%%%%%%%%%%%%%

%%%%%%%%%%%%%%%%%%%%%%%
\section{The CHB system in dimension $N=2$ }
\label{CHB2D}
 In  this section, we analyze the closeness between the  solution to the CHB system and the solution  to the CHHS system which are
 originated from regular initial data in $\H^2$.

Before proving our main result, i.e., Theorem \ref{t:close}, we derive some
regularity estimates for the solutions of the CHB system in 2D which are
uniform with respect to $\nu\geq 0$.
Hence, from now on, let $\phi_0\in \H^2$ and denote by
$c\geq 0$ a generic constant which may depend on $\|\phi_0\|_2$ but is \emph{independent of $\nu$.}

\subsection{Higher-order bounds independent of $\nu$}
We shall exploit in a crucial way the following well-known inequalities which hold in dimension
two:
\begin{align}
\label{L4m}
&\|f\|_{L^4}^2\leq c(\|f\|\|\nabla f\|+\|f\|^2),\\
\label{L4}
&\|f\|_{L^4}^2\leq c\|f\|\|\nabla f\|,\quad \text{if } \l f \r=0,\\
\label{A2}
&\|f\|_{L^\infty}^2\leq c\|f\|\|f\|_{H^2}.
\end{align}

\begin{proposition}
Let $\nu\geq 0$ be fixed and let $\phi(t)=S_\nu(t)\phi_0$. Then, the following estimate holds
\begin{equation}
\label{puntualeH2_2d}
\|\phi(t)\|_2+\int_t^{t+1}\|\phi(y)\|_4^2\,\d y\leq c,\quad \forall t\geq 0.
\end{equation}
Furthermore, we have
\begin{equation}\label{uH1}
    \sup_{t\geq 0}\int_t^{t+1}(\|\mu(y)\|_2^2 )\,\d y\leq c.
\end{equation}

\end{proposition}

\begin{proof}
On account of \eqref{Stokes-id} we find
$$\nu\|\nabla\u\|^2+\|\u\|^2=\l \mu\nabla\phi,\u\r\leq \frac12 \|\u\|^2+\frac12\|\mu\nabla\phi\|^2,$$
which yields
\begin{equation*}
\|\u\|^2\leq \|\mu\nabla\phi\|^2.
\end{equation*}
Besides, by \eqref{L4m} and \eqref{L4} we get
$$
\|\mu\nabla\phi\|^2\leq \|\mu\|_{L^4}^2\|\nabla\phi\|_{L^4}^2
\leq c (\|\mu\| \|\nabla \mu\|+\|\mu\|^2) \|\nabla\phi\| \|\Delta\phi\|.
$$
Since standard computations in light of \eqref{GROW} and \eqref{e1} yield
\begin{align*}
\|\mu\|&\leq \|\Delta \phi\|+\|f(\phi)\|\leq c(1+\|\Delta \phi\|),\\
\|\nabla\mu\|&\leq \|\nabla\Delta \phi\|+\|\nabla f(\phi)\|\leq c(1+\|\nabla\Delta \phi\|),
\end{align*}
we end up with
\begin{equation}
\label{uf}
\|\u\|^2\leq c(1+\|\Delta \phi\|^2) \|\nabla\Delta\phi\|.
\end{equation}
By taking $w=\Delta^2\phi$ in \eqref{B1-w} we obtain
$$
\frac12\ddt \|\Delta\phi\|^2+\|\Delta^2 \phi\|^2=\l\Delta f(\phi),\Delta^2\phi\r+\l\u\cdot\nabla\phi,\Delta^2\phi\r.
$$
We estimate the first term on the right hand side as follows
\begin{align*}
 \l\Delta f(\phi),\Delta^2\phi\r
\leq \frac{1}{4}\|\Delta^2\phi\|^2+ c\|\Delta f(\phi)\|^2
\leq \frac{1}{4}\|\Delta^2\phi\|^2+c(1+\|\Delta \phi\|^2)\|\Delta\phi\|^2,
\end{align*}
where we exploit the 2D analog of \eqref{Delta_f} to control $\|\Delta f(\phi)\|$.
Then we handle the remaining term as
\begin{align*}
\l\u\cdot\nabla\phi,\Delta^2\phi\r\leq c\|\u\cdot\nabla\phi\|\|\Delta^2\phi\|
&\leq \frac{1}{4}\|\Delta^2\phi\|+ c\|\u\cdot\nabla\phi\|^2.
\end{align*}
Owing to \eqref{uf} and the Agmon inequality \eqref{A2}, we infer
\begin{align*}
&\|\u\cdot\nabla\phi\|^2\leq \|\u\|^2\|\nabla\phi\|_{L^\infty}^2
\leq \|\u\|^2 \|\nabla\phi\| \|\nabla\Delta\phi\|
\leq c(1+\|\Delta \phi\|^2)\|\nabla\Delta\phi\|^2\\
&\quad\leq  c\|\Delta \phi\|^2\|\nabla\Delta\phi\|^2+c\|\Delta\phi\|\|\Delta^2\phi\|
\leq  c(1+\|\nabla\Delta\phi\|^2)\|\Delta \phi\|^2+\frac14\|\Delta^2\phi\|^2.
\end{align*}
Thus we obtain the differential inequality
\begin{equation}
\label{dif_fi2}
\frac12\ddt \|\Delta\phi\|^2+\frac12\|\Delta^2 \phi\|^2\leq
 g(t)\|\Delta\phi\|^2,
\end{equation}
where, in light of \eqref{L2-H3},
$g(t):=c(1+\|\Delta \phi(t)\|^2+\|\nabla\Delta\phi(t)\|^2)$
satisfies
$$\sup_{t\geq 0}\int_t^{t+1}g(y)\,\d y\leq c.$$
We can thus apply Lemma \ref{gen-gronw}, so obtaining
$$\|\Delta\phi(t)\|^2\leq c,\quad \forall t\geq 1.$$
In order to prove the required estimate for $t\in [0,1]$ it is sufficient to apply the usual Gronwall lemma on $[0,t]$ to the inequality
$$\ddt \|\Delta\phi\|^2\leq
 2g(t)\|\Delta\phi\|^2.$$
Indeed this yields
$$\|\Delta\phi(t)\|^2\leq \|\Delta\phi(0)\|^2\e^{2G(t)},$$
where
$$G(t)=\int_0^t g(y)\,\d y \leq \int_0^1 g(y)\,\d y\leq c,\quad \forall t\in[0,1].$$
Hence we have
$$\|\Delta\phi(t)\|^2\leq c,\quad \forall t\in[0,1].$$
On account of this bound, a final integration of \eqref{dif_fi2} on $[t,t+1]$ concludes the proof of~\eqref{puntualeH2_2d}. In order to show the validity of \eqref{uH1},
note that, by estimating again $\|\Delta f(\phi)\|$ as in \eqref{Delta_f}, we get
\begin{align*}
\|\mu\|_2^2&\leq c(\|\mu\|^2+\|\Delta\mu\|^2)\\&\leq c(\|f(\phi)\|^2+\|\Delta\phi\|^2+\|\Delta f(\phi)\|^2+\|\Delta^2\phi\|^2)\\
&\leq c(1+\|\Delta \phi\|^2+\|\Delta \phi\|^4+\|\Delta^2 \phi\|^2),
\end{align*}
which, in light of \eqref{puntualeH2_2d}, implies the integrability of $\mu$.
\end{proof}

\begin{remark}\label{PM_forcing_higher_est}
    The following estimate also holds uniformly in $\nu \geq 0$. Exploiting the Agmon inequality and the uniform $\H^2$-estimate for $\phi$, we have
    \begin{align*}
    \| \mu \nabla \phi \|_1^2& \leq \|\mu\nabla\phi\|^2+\|\nabla\mu\nabla\phi\|^2+\|\mu\nabla^2\phi\|^2\\
    &\leq \|\mu\|_{L^\infty}^2\|\nabla\phi\|^2+\|\nabla\mu\|_{L^4}^2\|\nabla\phi\|_{L^4}^2+c\|\mu\|_{L^\infty}^2\|\Delta\phi\|^2\\
    &\leq c(1+\|\mu\|_2^2).
    \end{align*}
    In particular we deduce $\mu \nabla \phi \in L^{2}(t,t+1;\H^1)$ uniformly for $t \geq 0$ and $\nu \geq 0$.
\end{remark}

%%%%%%%%%%%%%%%%%%%%%%%%%%%%%%%%

%%%%%%%%%%%%%%%%%%%%%%%%%%%%%%%%

\subsection{Proof of Theorem \ref{t:close}}

\begin{proof}
Let $\phi_0^\nu,\phi_0\in\H^2$ such that $\l\phi_0^\nu\r=\l\phi_0\r$.
Then denote by $c$ a generic positive constant depending on $R$, where
$R:=\sup_{\nu> 0}\{\|\phi_0^\nu\|_2,\|\phi_0\|_2\}<\infty.$
Let $(\phi_\nu,\u_{\nu})$ be the weak solution to the CHB system with $\nu>0$ originating from $\phi_0^{\nu}$, and $(\phi, \u)$ the solution to the CHHS system with initial datum $\phi_0$.
Note that the difference  $\bar \phi=\phi_\nu-\phi$, $\bar\u =\u_\nu-\u$ is a weak solution to
\begin{align}
\label{D1_2d}
&\partial_t \bar\phi+\nabla\cdot (\phi_\nu
\bar\u)+\nabla\cdot (\bar \phi \u)-\Delta \bar\mu=0,\\
\label{D2_2d}
&\bar\u=\nabla \bar p-\phi_\nu \nabla \bar\mu-
 \bar\phi \nabla \mu+ \nu \Delta \u_\nu,\\
\label{Ddiv_2d}
&\nabla \cdot   \bar\u=0,
\end{align}
where
\begin{equation*}\label{Dmu_2d}
    \bar \mu =-\Delta \bar\phi+[f(\phi_\nu)-f(\phi)],
\end{equation*}
and $\l\bar\phi\r=0$.

\medskip
\noindent
Taking  $-\Delta \bar\phi$ as test function in the weak formulation of \eqref{D1_2d}, we obtain
$$\ddt \frac12 \|\nabla \bar\phi\|^2+\l \phi_\nu\bar\u,\nabla\Delta \bar\phi\r
+\l \bar\phi\u,\nabla\Delta \bar\phi\r+\l \nabla \bar\mu,\nabla\Delta \bar\phi\r=0.$$
On the other hand, we have
$$\l \nabla \bar\mu,\nabla\Delta \bar\phi\r=-\|\nabla\Delta \bar\phi\|^2+
\l \nabla[f(\phi_\nu)-f(\phi)],\nabla\Delta \bar\phi\r,$$
so that
\begin{equation}
\label{uno_2d}\ddt \frac12 \|\nabla \bar\phi\|^2+\|\nabla\Delta \bar\phi\|^2=-\l \phi_\nu\bar\u,\nabla\Delta \bar\phi\r -\l \bar\phi\u,\nabla\Delta \bar\phi\r+
\l \nabla[f(\phi_\nu)-f(\phi)],\nabla\Delta \bar\phi\r.
\end{equation}
Let us now take $\bar \u$ in the weak formulation of \eqref{D2_2d}. Adding $-\nu \l \nabla \u, \nabla \bar\u \r$ to both sides of the resulting identity, we get
\begin{equation}
\label{due}
\nu \| \nabla \bar \u \|^2 + \|\bar \u\|^2=-\l \phi_\nu\nabla\bar\mu,\bar\u\r-\l \bar \phi\nabla\mu,\bar\u\r
-\nu\l \nabla \u,\nabla \bar\u\r.
\end{equation}
Note that, by definition of $\bar\mu$, there holds
$$-\l \phi_\nu\nabla\bar\mu,\bar\u\r=\l \phi_\nu\nabla\Delta\bar\phi,\bar\u\r-
\l \phi_\nu\nabla[f(\phi_\nu)-f(\phi)],\bar\u\r.$$
Hence, adding  \eqref{uno_2d} with \eqref{due}
we end up with
\begin{align*}
&\ddt \frac12 \|\nabla \bar\phi\|^2+\|\nabla\Delta \bar\phi\|^2 + \nu \| \nabla \bar \u \|^2 + \|\bar \u\|^2=-\nu\l \nabla\u,\nabla \bar\u\r\\
&\quad -\l \bar\phi\u,\nabla\Delta \bar\phi\r - \l \bar \phi\nabla\mu,\bar\u\r - \l \phi_\nu\nabla[f(\phi_\nu)-f(\phi)],\bar\u\r+\l \nabla[f(\phi_\nu)-f(\phi)],\nabla\Delta \bar\phi\r.
\end{align*}
We now estimate the terms on the right hand side.
First of all, we have
\begin{align*}
-\l \bar\phi\u,\nabla\Delta \bar\phi\r \leq c \|\bar\phi\|_1 \|\u\|_{1} \|\nabla\Delta \bar\phi\|
\leq \frac{1}{4}\|\nabla\Delta \bar\phi\|^2 + c \|\u\|_{1}^2\|\bar\phi\|_1^2.
\end{align*}
Besides, the following inequality holds
\begin{align*}
-\l \bar \phi\nabla\mu,\bar\u\r \leq \|\bar\phi\|_1 \|\bar{\u}\| \|\nabla \mu\|_{L^3}
\leq \frac{1}{2} \|\bar\u\|^{2} + c\|\Delta \mu\|^2\|\bar\phi\|_1^2.
\end{align*}

\smallskip
\noindent
We are left to deal with the term
\begin{align*}
\l \nabla[f(\phi_\nu)-f(\phi)],\nabla\Delta \bar\phi\r&\leq \|\nabla[f(\phi_\nu)-f(\phi)]\|\|\nabla\Delta \bar\phi\|\\&\leq \frac{1}{4}\|\nabla\Delta \bar\phi\|^2+c\|\nabla[f(\phi_\nu)-f(\phi)]\|^2,
\end{align*}
where
$$\|\nabla[f(\phi_\nu)-f(\phi)]\|^2\leq\|[f'(\phi_\nu)-f'(\phi)]\nabla \phi_\nu\|^2+ \|f'(\phi)\nabla\bar\phi\|^2.$$
By exploiting the uniform $\H^2$-estimates both for $\phi_\nu$ and $\phi$ obtained in \eqref{puntualeH2_2d} and condition \eqref{GROW2}, we have
\begin{align*}
\|f'(\phi)\nabla\bar\phi\|^{2}&=\int_\Omega |f'(\phi)\nabla\bar\phi|^{2}
\leq \|f'(\phi)\|_{L^\infty}^2\|\nabla\bar\phi\|^{2}
\leq c(1+\|\phi\|_{L^\infty}^4)\|\nabla\bar\phi\|^{2}
\leq c \|\bar\phi\|_1^2,
\end{align*}
and, analogously,
\begin{align*}
\|[f'(\phi_\nu)-f'(\phi)]\nabla \phi_\nu \|^{2}&\leq c\int_\Omega |(1+|\phi_\nu|+|\phi|)\bar\phi\nabla\phi_\nu|^{2}
\leq c\|\bar\phi\|_{L^4}^2\|\nabla \phi_\nu\|_{L^4}^2
\leq c \|\bar\phi\|_1^2.
\end{align*}
Thus we have the control
$$\|\nabla[f(\phi_\nu)-f(\phi)]\|^2\leq c \|\bar\phi\|_1^2.$$
Using again \eqref{puntualeH2_2d}, the remaining term involving $f$ can be treated in the following way:
\begin{align*}
-\l \phi_\nu\nabla[f(\phi_\nu)-f(\phi)],\bar\u\r&\leq \|\phi_\nu\|_{L^\infty} \|\nabla[f(\phi_\nu)-f(\phi)]\| \|\bar \u\|\\
&\leq \frac{1}{4} \|\bar{\u}\|^2+ c\|\nabla[f(\phi_\nu)-f(\phi)]\|^2
\leq \frac{1}{4} \|\bar{\u}\|^2+c \|\bar\phi\|_1^2.
\end{align*}
In addition, we have
$$\nu|\l \nabla\u,\nabla \bar\u\r|\leq \nu \| \nabla \u \|^2 + \nu | \l \nabla \u, \nabla \u_{\nu} \r | \leq \nu \| \nabla \u \|^2 + \nu^{1/2} ( \nu \| \nabla \u_{\nu} \|^2 + \| \nabla \u \|^2 ).$$
Thanks to~\cite[Lemma~2.1]{LTZ}, Remark~\ref{PM_forcing_higher_est} implies $\u \in L^2 (t,t+1;\HD)$ for all $t \geq 0$. Moreover, recalling \eqref{e-base}, it holds
$$k(\cdot):= \nu \| \nabla \u \|^2 + (1 + \nu^{1/2}) \| \nabla \u \|^2 \in L^1(0,T),$$
for every $T>0$, uniformly with respect to $\nu\geq 0$.

\smallskip
\noindent
Collecting all the above inequalities, we end up with
$$\ddt  \|\bar\phi\|_1^2 + \frac{1}{4}\|\bar{\u}\|^{2} \leq h(t)\|\bar\phi\|_1^2+\nu^{1/2} k(t),$$
where $h(t)=c(1+\|\Delta \mu(t)\|^2+\|\u(t)\|_{1}^2).$
Thanks again to~\cite[Lemma~2.1]{LTZ} and Remark~\ref{PM_forcing_higher_est}, $h\in L^1(0,T)$ uniformly with respect to $\nu\geq 0$.
Therefore, an application of the Gronwall lemma provides, for all $t\in [0,T]$,
$$\|\phi_\nu(t)-\phi(t)\|_1^2\leq
\|\phi_0^\nu-\phi_0\|_1^2 \e^{\int_0^t h(y)\,\d y}+ \nu^{1/2} \int_0^t k(y)\,\d y,$$
which, in particular, entails that
$$\|\phi_\nu(t)-\phi(t)\|_1^2\leq
\|\phi_0^\nu-\phi_0\|_1^2\e^{C_T} + \nu^{1/2} \, C_T,$$
having set $C_T=\max\{\int_0^T h(y)\,\d y ,\, \int_0^T k(y)\,\d y\}<\infty$.
Integrating the differential inequality on $[0,t]$, $t\leq T$, up to enlarging $C_T$ we also obtain
\begin{equation*}
\int_{0}^{t} \|\u_\nu-\u \|^{2} \leq \|\phi_0^\nu-\phi_0\|_1^2\e^{C_T}+ \nu^{1/2}\, C_T. \qedhere
\end{equation*}
\end{proof}

\begin{remark}
Since we are dealing with solutions which are uniformly bounded in $\H^2$, the convergence of $\phi_\nu$ to $\phi$ in $\H^{2-\delta}$, for every $\delta>0$, easily follows. On the contrary, proving the convergence in $\H^2$ seems to be out of reach, due to the fact that the semigroup associated to the solutions of the CHHS equation on $\H^2$ is not strongly continuous but just closed, with a continuous dependence estimate with respect to the $\H^1$-norms, see \cite[(6.13)]{LTZ}.
\end{remark}
%%%%%%%%%%%%%%%%%%%%%%%%%%%%%%%%

{\bf Acknowledgments.} The authors are grateful to the referees for their helpful comments and suggestions. The three authors are members of the Gruppo Nazionale per l'Analisi Matematica, la Probabilit\`{a} e le loro Applicazioni (GNAMPA) of the Istituto Nazionale di Alta Matematica (INdAM).

%%%%%%%%%%%%%%%%%%%%%%%%%%%%%%%%

\end{document}